\documentclass{amsart}

\usepackage{amssymb,amsmath,amsthm,amsfonts,framed}
\usepackage{framed,comment,enumerate}
\usepackage{xcolor, framed}
\usepackage{tikz}
\usepackage{color, framed}
\usepackage{graphicx,caption,subcaption}

\usepackage[colorlinks]{hyperref}
\hypersetup{linkcolor=blue,citecolor=blue,filecolor=black,urlcolor=blue}

\usepackage{verbatim}

\usepackage[showonlyrefs]{mathtools}
\usepackage[sort,nocompress]{cite}

\usepackage[
  hmarginratio={1:1},     % equal left and right margins
  vmarginratio={1:1},     % equal top and bottom margins
  textwidth=15cm,        % new text width
  textheight=21cm,
  heightrounded,          % always useful
]{geometry}
\def\R {\mathbb{R}}

\def\D {\mathcal{D}}
\def\N {\mathbb{N}}

\def\eps{\varepsilon}
\def\dist{{\rm dist}}

\newcommand{\cA}{{\mathcal A}}
\newcommand{\cB}{{\mathcal B}}

\newcommand{\cD}{{\mathcal D}}
\newcommand{\cE}{{\mathcal E}}

\newcommand{\cG}{{\mathcal G}}

\newcommand{\cP}{{\mathcal P}}

\newcommand{\cS}{{\mathcal S}}
\newcommand{\cT}{{\mathcal T}}

\newcommand{\be}{\beta}

\newcommand{\la}{\lambda}
\newcommand{\si}{\sigma}

\newcommand{\De}{\Delta}

\newcommand{\loc}{\mathrm{loc}}

\newcommand{\wc}{\rightharpoonup}

\newcommand{\pa}{\partial}

\DeclareMathOperator{\rad}{rad}

\newtheorem{proposition}{Proposition}[section]
\newtheorem{theorem}[proposition]{Theorem}
\newtheorem*{thm A}{Theorem A}
\newtheorem*{theorem*}{Theorem}

\newtheorem{lemma}[proposition]{Lemma}

\theoremstyle{definition}
\newtheorem{definition}[proposition]{Definition}

\newtheorem{remark}[proposition]{Remark}
\numberwithin{equation}{section}
\newenvironment{altproof}[1]
{\noindent%\addvspace{0.3cm}
{\em Proof of {#1}}.}
{\nopagebreak\mbox{}\hfill $\Box$\par\addvspace{0.5cm}}

\title[Multiple normalized solutions]{Multiple normalized solutions for a competing system of Schr\"odinger equations}

\author{Thomas Bartsch and Nicola Soave}

\address{
\hbox{\parbox{5.7in}{\medskip\noindent
Thomas Bartsch\\
Mathematisches Institut, Justus-Liebig-Universit\"at Giessen, \\
Arndtstrasse 2, 35392 Giessen (Germany).\\[2pt]
{\em{E-mail address: }}{\tt Thomas.Bartsch@math.uni-giessen.de.} \\ [5pt]
Nicola Soave\\
Dipartimento di Matematica, Politecnico di Milano, \\
Piazza Leonardo da Vinci, 32, 20133 Milano (Italy). \\[2pt]
{\em{E-mail address: }}{\tt nicola.soave@gmail.com, nicola.soave@polimi.it}}}}

\keywords{Elliptic systems, Schr\"odinger systems, Natural constraint, min-max methods.}

\subjclass[2010]{35J50, 35J15, 35J60.}

\thanks{\em{Acknowledgements:} Nicola Soave is partially supported by the project ERC Advanced Grant 2013 n. 339958 ``Complex Patterns for Strongly Interacting Dynamical Systems - COMPAT'', by the PRIN-2015KB9WPT\texttt{\char`_}010 Grant: ``Variational methods, with applications to problems in mathematical physics and geometry", and by the GNAMPA group.}

\begin{document}

\begin{abstract}
We prove the existence of infinitely many solutions $\lambda_1, \lambda_2 \in \R$, $u,v \in H^1(\R^3)$, for the nonlinear Schr\"odinger system
\[
\begin{cases}
-\Delta u - \lambda_1 u = \mu  u^3+ \beta u v^2 & \text{in $\R^3$} \\
-\Delta v- \lambda_2 v =  \mu v^3 +\beta u^2 v & \text{in $\R^3$}\\
u,v>0 & \text{in $\R^3$} \\
\int_{\R^3} u^2 = a^2 \quad \text{and} \quad \int_{\R^3} v^2 = a^2,
\end{cases}
\]
where $a,\mu>0$ and $\beta \le -\mu$ are prescribed. Our solutions satisfy $u\ne v$ so they do not come from a scalar equation. The proof is based on a new minimax argument, suited to deal with normalization conditions.
% For the proof we adapt the Krasnoselkii genus approach, which has to be substantially refined in order to face the strong lack of compactness of the problem.
\end{abstract}

\maketitle

\section{Introduction}

In this paper we consider the stationary nonlinear Schr\"odinger system
\begin{equation}\label{system}
\begin{cases}
-\Delta u - \lambda_1 u = \mu_1  u^3+ \beta u v^2 & \text{in $\R^3$} \\
-\Delta v- \lambda_2 v =  \mu_2 v^3 +\beta u^2 v & \text{in $\R^3$}\\
u,v>0 & \text{in $\R^3$} \\
\int_{\R^3} u^2 = a_1^2 \quad \text{and} \quad \int_{\R^3} v^2 = a_2^2,
\end{cases}
\end{equation}
where $a_1,a_2,\mu_1,\mu_2>0$ and $\beta<0$ are prescribed and $u,v \in H^1(\R^3)$, $\lambda_1, \lambda_2 \in \R$ have to be determined.

This problem possesses many physical motivations, e.g.\ it appears in models for nonlinear optics and Bose-Einstein condensation (we refer to \cite{BaSo} and the references therein for a more exhaustive discussion). Due to the physical background, it seems natural to search for normalized solutions (i.e.\ solutions with prescribed $L^2$-norm), but despite this fact most of the papers regarding \eqref{system} deal with the system with fixed frequencies (i.e.\ $\lambda_1,\lambda_2<0$ are prescribed, and the $L^2$-constraints are neglected), and not much is known about the full problem \eqref{system}. The only results available in the setting considered here are presented in \cite{BaJeSo, BaSo}, where for possibly non-symmetric systems we proved existence of one positive radial normalized solution, both for suitable choices of $\beta > 0$ \cite{BaJeSo}, and for all $\beta<0$ \cite{BaSo}. In this paper we consider the symmetric problem \eqref{system} with $\mu_1=\mu_2$ and $a_1 = a_2$ and, exploiting the symmetry, we prove the existence of infinitely many solutions, which will be found as critical points of the \emph{energy} functional $J_\be: \mathcal{S} \to \R$, defined by
\[%begin{equation}\label{def J}
  J_\be(u,v)
   := \frac12\int_{\R^3} |\nabla u|^2 + |\nabla v|^2 -\frac14 \int_{\R^3} \mu_1 u^4  + 2 \beta u^2 v^2 +  \mu_2 v^4,
\]%end{equation}
with
\[
  \mathcal{S} := S_{a_1} \times S_{a_2}, \quad \text{and}
  \quad S_a:= \left\{ w \in H_{\rad}^1(\R^3): \int_{\R^3} w^2 = a^2\right\}.
\]
Here $H^1_{\rad}(\R^3)$ denotes the space of radially symmetric functions in $H^1(\R^3)$. In this perspective, $\lambda_1$ and $\lambda_2$ arise as Lagrange multipliers with respect to the mass constraint. Clearly, if $(u,v,\la_1,\la_2)$ solves \eqref{system} then so does $(v,u,\la_2,\la_1)$.

\begin{theorem}\label{thm: main}
Let $a, \mu>0$, and let us consider system \eqref{system} with $a_1=a_2=a$ and $\mu_1=\mu_2 = \mu$. Then for any $k\in\N$ there exists $\be_k>-\mu$ such that for $\be<\be_k$ the problem \eqref{system} has at least $k$ different pairs $(u_{j,\beta},v_{j,\beta}, \lambda_{1,\beta}^j, \lambda_{2,\beta}^j)$, $(v_{j,\beta},u_{j,\beta}, \lambda_{2,\beta}^j, \lambda_{1,\beta}^j)$, $j=1,\dots,k$, of radial solutions with increasing energy. The solutions satisfy $u_{j,\beta}\ne v_{j,\beta}$, and $J_\beta(u_{j,\beta},v_{j,\beta})\to\infty$ as $j\to\infty$, provided $\be\le-\mu$.
\end{theorem}

We can also show that, for any $k \in \N$ fixed, the family $\{(u_{k,\beta},v_{k,\beta}): \beta \le -\mu\}$ segregates in the limit of strong competition:

\begin{theorem}\label{thm: phase sep}
Let $k \in \N$. As $\beta \to -\infty$, up to a subsequence we have:
\begin{itemize}
\item[($i$)] $(\lambda_{1,\beta}^k,\lambda_{2,\beta}^k) \to (\lambda_1^k,\lambda_2^k)$, with $\lambda_1^k,\lambda_2^k \le 0$;
\item[($ii$)] $(u_{k,\beta},v_{k,\beta}) \to (u_k,v_k)$ in $\mathcal{C}^{0,\alpha}_{\loc}(\R^N)$ and in $H^1_{\loc}(\R^N)$, for any $\alpha \in (0,1)$;
\item[($iii$)] $u_k$ and $v_k$ are nonnegative Lipschitz continuous functions having disjoint positivity sets, in the sense that $u_k v_k \equiv 0$ in $\R^N$;
\item[($iv$)] the difference $u_k- v_k$ is a sign-changing radial solution of
\[
  -\Delta w -\lambda_1^k w^+ +\lambda_2^k w^- = \mu w^3 \qquad \text{in $\R^3$}.
\]
\end{itemize}
\end{theorem}

\begin{remark}
The scalar problem
\begin{equation}\label{syst scalar}
\begin{cases}
  -\De w -\la w = \mu w^3 & \text{in $\R^3$} \\
  w \in S_a
\end{cases} \quad \text{for some $\la <0$};
\end{equation}
has a unique positive radial solution $w_0\in H^1_{\rad}(\R^3)$. Setting
\[
  w_\be(x):=\left(\frac{\mu}{\mu+\be}\right)^{\frac32}w_0\left(\frac{\mu}{\mu+\be}x\right),\quad
  \la_\be:=\left(\frac{\mu}{\mu+\be}\right)^2,
\]
for $\be>-\mu$ we obtain a smooth curve
\[
  \cT:=\left\{(w_\be,w_\be,\la_\be,\la_\be):\be>-\mu\right\}
\]
of symmetric solutions of \eqref{system}. This suggests that the solutions in Theorem~\ref{thm: main} bifurcate from $\cT$ as in \cite{BaDaWa}. We do not pursue this approach here.
\end{remark}

In what follows we recall basic facts concerning the existence of normalized solutions for nonlinear Schr\"odinger equations in $\R^N$ and describe the strategy of the proof of Theorem \ref{thm: main}, emphasizing the main differences with respect to the results already present in the literature. This serves also as motivation to our study.

The homogeneous nonlinear Schr\"odinger equation with normalization constraint is
\begin{equation}\label{eq intro}
  -\Delta w -\lambda w = |w|^{p-2} w \qquad \text{in $\R^N$}, \quad \int_{\R^N} w^2 = a^2.
\end{equation}
It is well known that two exponents play a special role for existence and properties of the solutions: in addition to the \emph{Sobolev critical} exponent $p=2N/(N-2)$, we have the \emph{$L^2$-critical} one $p=2+4/N$. If $2<p<2+4/N$ ($L^2$-subcritical regime), then the energy functional associated to \eqref{eq intro} is bounded from below on the $L^2$-sphere $S_a$, while if $p \ge 2+4/N$ ($L^2$-critical or supercritical regime) this is not true and one is forced to search for critical points that are not global minima. The critical Sobolev exponent defines the threshold for the existence of a $H^1$-solution. The very same discussion applies to systems of type \eqref{system}, and this is why our results concern the space dimension $N=3$: since we are considering cubic nonlinearities, in dimension $N=1,2,3$ or $4$ we have respectively a $L^2$-subcritical, $L^2$-critical, $L^2$-supercritical and Sobolev-subcritical, or Sobolev-critical setting, and each framework requires its own techniques. With regard to this, we mention that while for $L^2$-subcritical problems many results are available (see e.g. \cite{Stu1, Stu2, Lio1,Lio2} for equations and \cite{Caoetal,GuoJea, NguWan} for systems), the $L^2$-critical or supercritical ones are much less understood, and we refer to \cite{Je, BaDVa} for equations and to \cite{BaJeSo, BaJe, BaSo} for systems.

Let us focus now on system \eqref{system} in the symmetric case $a_1=a_2$ and $\mu_1=\mu_2$. Since the problem is invariant both under rotations, and with respect to the involution $\sigma:(u,v) \mapsto (v,u)$, it is natural to adapt the Krasnoselskii genus approach to the constrained functional $J_\beta|_{\mathcal{S}}$. This is the strategy used in \cite{NoTaTeVeJEMS}, where the authors considered normalized solution to \eqref{system} in the case when $\mu<0$ and $\R^3$ is replaced by a bounded domain $\Omega$ (with homogeneous Dirichlet boundary conditions). In such a situation the functional is bounded from below, coercive, and satisfies the Palais-Smale condition on the product of the $L^2$-spheres\footnote{Indeed, in \cite{NoTaTeVeJEMS} the authors exploit a uniform-in-$\beta$ Palais-Smale condition to derive the convergence of the whole minimax structure to a limit problem.}. All these properties, which are essential to use the Krasnoselskii genus, fail when considering \eqref{system} in $\R^3$ with $\mu>0$: the functional $J_\beta$ is indeed unbounded both from above and from below on $\cS$, and the Palais-Smale condition is not satisfied. In \cite{DaWeWe}, where system \eqref{system} is studied in the case of fixed frequencies $\lambda_1,\lambda_2<0$, the same complications are overcome with the introduction of a Nehari-type manifold $\mathcal{N}_\beta$ associated to the problem. The authors proved that the constrained functional $J_\beta|_{\mathcal{N}_\beta}$ is bounded from below, coercive, and satisfies the Palais-Smale condition.

Searching for normalized solutions the Nehari manifold is not available, but in \cite{BaSo} we introduced a different additional constraint, suited to treat problems with normalization conditions:
\[
  \cP_\be := \left\{ (u,v) \in \mathcal{S}:
      \int_{\R^3} |\nabla u|^2 + |\nabla v|^2 = \frac34 \int_{\R^3} \mu_1 u^4  + 2\beta u^2 v^2 + \mu_2 v^4 \right\}.
\]
This is a $\mathcal{C}^2$ submanifold of $\mathcal{S}$, see \cite[Lemma 2.2]{BaSo}\footnote{In \cite{BaSo} we showed that $\cP_\be$ is a $\mathcal{C}^1$ manifold since this was enough for our purpose, the extra regularity is straightforward.}, and using the Pohozaev identity it is easy to check that any weak solution to \eqref{system} stays in $\cP_\be$. The manifold $\cP_\be$ is a \emph{natural constraint}, i.e.\ if $(u,v)\in\cP_\be$ is a critical point of $J|_{\cP_\be}$ then it is a critical point of $J$ on $\cS$. In \cite[Theorem 2.1]{BaSo}, we also stated that a constrained Palais-Smale sequences for $J_\beta$ on $\cP_\be$ gives rise to a ``free" Palais-Smale sequences for $J_\beta$ on $\cS$; however, the proof of \cite[Theorem 2.1]{BaSo}, and analogously the one of \cite[Theorem 4.1]{BaSo}, contains a gap. In the present paper we show that a minimax value for the constrained functional $J|_{\cP_\be}$ yields a Palais-Smale sequence for $J|_\cS$ consisting of elements in $\cP_\be$. This is slightly weaker than \cite[Theorem 4.1]{BaSo} but sufficient for our purposes here, and also for the main results from \cite{BaSo}.

We need some notation first. Let $X \subset H^1_{\rad}(\R^3,\R^2)$ and, as above, let $\sigma(u,v) = (v,u)$. A set $A \subset X$ is $\sigma$-invariant if $\sigma(A) = A$; similarly, a function $f:X \to \R$ is called $\sigma$-invariant if $f(\sigma(u,v)) = f(u,v)$ for every $(u,v)$. A continuous function $h:X \to X$ is $\sigma$-equivariant if $h(\sigma(u,v)) = \sigma(h(u,v))$. A homotopy $\eta:[0,1] \times X \to X$ is $\sigma$-equivariant if $\eta(t,\cdot)$ is $\sigma$-equivariant for any $t \in [0,1]$.

Notice that both $J_\be$ and $\cP_\be$ are $\sigma$-invariant, under the assumption $a_1=a_2$ and $\mu_1=\mu_2$.

\begin{definition}\label{def hsf equi}
Let $B$ be a closed $\si$-invariant subset of $X$. We say that a class $\mathcal{F}$ of compact subsets of $X$ is a \emph{$\sigma$-homotopy stable family with closed boundary $B$} provided:
\begin{itemize}
\item[($a$)] every set in $\mathcal{F}$ contains $B$.
\item[($b$)] for any $A \in \mathcal{F}$ and any $\sigma$-equivariant homotopy $\eta \in C([0,1] \times X,X)$ satisfying $\eta(t,x) = x$ for all $(t,x) \in (\{0\} \times X) \cup ([0,1] \times B)$, we have that $\eta(\{1\} \times A) \in \mathcal{F}$.
\item[($c$)] every set in $\mathcal{F}$ is $\sigma$-invariant.
\end{itemize}
\end{definition}

\begin{theorem}\label{thm:PSequi}
Let $\beta \in \R$, $a_1=a_2$, $\mu_1=\mu_2$, and let $\mathcal{F}$ be a $\sigma$-homotopy stable family of compact subsets of $\mathcal{P}_\beta$, with closed boundary $B \subset \mathcal{P}_\beta$. Let
\[
c_{\mathcal{F},\beta}:= \inf_{A \in \mathcal{F}} \, \max_{(u,v) \in A} J_\beta(u,v)
\]
Suppose that
\begin{equation}\label{hp minmax equi}
\text{$B$ is contained in a connected component of $\mathcal{P}_\beta$},
\end{equation}
and that
\begin{equation}\label{hp minmax2 equi}
\max\{\sup J_\beta(B),0\} <  c_{\mathcal{F},\beta} < +\infty.
\end{equation}
%and that
%\begin{equation}\label{invariancy}
%\text{J_\beta$ is $\sigma$-invariant}.
%\end{equation}
Then there exists a sequence $\{(u_n,v_n)\}$ with the following properties:
\begin{itemize}
\item[($i$)] $(u_n,v_n) \in \mathcal{P}_\beta$ for every $n$;
\item[($ii$)] $J_\beta(u_n,v_n) \to c_{\mathcal{F},\beta}$ as $n \to \infty$;
\item[($iii$)] $\|\nabla (J_\beta|_{\cS})(u_n,v_n)\| \to 0$ as $n \to \infty$, i.e. $\{(u_n,v_n)\}$ is a Palais-Smale sequence for $J_\beta$ on $\cS$.
\end{itemize}
If moreover we can find a minimizing sequence $\{A_n\}$ for $c_{\mathcal{F},\beta}$ in such a way that $(u,v) \in A_n$ implies $u,v \ge 0$ a.e., then we can find the sequence $\{(u_n,v_n)\}$ satisfying the additional condition
\begin{itemize}
\item[($iv$)] $u_{n}^-, v_{n}^- \to 0$ a.e. in $\R^3$ as $n \to \infty$.
\end{itemize}
\end{theorem}

Here and in the rest of the paper, $\|\cdot\|$ denotes the $H^1(\R^3,\R^2)$ norm. We explicitly observe that $B=\emptyset$ is admissible, it is sufficient to adopt the usual convention $\sup(\emptyset) = -\infty$.

Theorem \ref{thm:PSequi} establishes that, if the assumptions of the equivariant minimax principle \cite[Theorem 7.2]{Ghou} are satisfied by the constrained functional $J_\be|_{\mathcal{P}_\be}$, then we can find a ``free" Palais-Smale sequence for $J_\be$ on $\mathcal{S}$, made of elements of $\cP_\beta$. The advantage of working with the constrained functional $J_\beta|_{\cP_\beta}$ stays in the fact that it has much better properties than $J_\beta|_{\cS}$: indeed, $J_\beta|_{\cP_\beta}$ is bounded from below and coercive. On the other hand, differently to what happen in the fixed frequency case \cite{DaWeWe}, the Palais-Smale condition is not satisfied on $\cP_\be$. This is a phenomenon purely related to the normalization conditions, and is motivated by the lack of compactness of the embedding $H^1_{\rad}(\R^3) \hookrightarrow L^2(\R^3)$. In particular, there exist Palais-Smale sequences for $J_\be|_{\cP_\be}$ converging weakly in $H^1(\R^3)$ (actually strongly in $\mathcal{D}^{1,2}(\R^3)$) to \emph{semi-trivial} bound states $(w,0)$ and $(0,w)$, where $w$ is a radial solution to \eqref{syst scalar}; we stress that $(w,0)$ and $(0,w)$ do not stay on $\cS$, hence not on $\cP_\be$.

As further complication, we observe that since $\lambda_1,\lambda_2$ are not prescribed and could be non-negative, the operators $-\Delta -\lambda_i$ could be both non-positive and non-invertible. This is particularly relevant here since we are interested in positive solutions, and not knowing the sign of $\lambda_i$ we cannot argue as in \cite{DaWeWe}, where the authors simply replaced the nonlinearity $\mu u^3$ with $\mu (u^+)^3$ in \eqref{system}, found non-trivial solutions of the new problem, and applied the maximum principle to obtain positive solutions of the original system.

In light of the previous discussion, for the proof of Theorem \ref{thm: main} we shall considerably refine the Krasnoselskii genus approach for $J_\beta|_{\cP_\beta}$, and in particular a careful analysis of the behavior of Palais-Smale sequences is needed.

\begin{remark}
An analogue of Theorem \ref{thm:PSequi} holds also in the non-symmetric setting, see Theorem \ref{thm:PS} in Section \ref{sec: nat}.
\end{remark}

\medskip

\noindent \textbf{Structure of the paper.} The proof of Theorem \ref{thm:PSequi} is the object of Section \ref{sec: nat} where we also treat the non-symmetric version \ref{thm:PS}.

Afterwards, we proceed with the proof of Theorem \ref{thm: main}. From Section \ref{sec: PS cond} on we always focus on the symmetric case $\mu_1=\mu_2=\mu$, $a_1=a_2=a$. In order to simplify some expressions and computations, we shall consider the case $\mu=1$, without loss of generality. We shall first show in Section \ref{sec: PS cond} that specially constructed Palais-Smale sequences do converge. In Section \ref{sec: many PS} we set up a minimax scheme producing infinitely many such sequences. Sections \ref{sec: prel} and \ref{sec: energy} are devoted to some auxiliary facts, and in Section \ref{sec: end} we complete the proofs of our main result.

%Without loss of generality, and in order to simplify some notation, we will always assume that $\mu=1$ from now on.

\medskip

To conclude the introduction, we mention that both \eqref{eq intro} and system \eqref{system} are studied also in bounded domains. When dealing with fixed frequencies, the search for radial solutions in $\R^N$, or for solutions satisfying homogeneous Dirichlet boundary conditions in $\Omega$ bounded, can be treated essentially in the same way. It is remarkable that, on the contrary, searching for normalized solutions the two problems ``$\R^N$" vs. ``bounded domains" presents substantial differences. In order to better understand this aspect, we invite the reader to compare the results and the techniques used here and in \cite{BaDVa, BaJeSo, BaSo, Je} in $\R^N$, and those in \cite{FibMer, NoTaVe2, NoTaVe, PiVe} in bounded domains.

\section{Notation and preliminaries}\label{sec: prel}

%Without loss of generality, from now on we assume $\mu=1$.
%As already announced, we aim at applying critical point theory to the functional $J_\be$ on $\cP_\be$.
%We sometimes write $J_\beta$ and $\mathcal{P}_\beta$ in order to emphasize the dependence with respect to the competition parameter.
We always work in the whole space $\R^3$, and hence we often write $H^1$ instead of $H^1(\R^3)$ and $H^1(\R^3,\R^2)$, to keep the notation as compact as possible. The same discussion applies to all the functional spaces we shall use. Regarding this, we recall that $\mathcal{D}^{1,2}(\R^3)$ denotes the closure of $\mathcal{C}^\infty_c(\R^3)$ with respect to norm $\|w\|_{\mathcal{D}^{1,2}}:= \|\nabla w\|_{L^2}$. We denote by $\|\cdot\|$ the $H^1(\R^3,\R^2)$ norm, and sometimes also the $H^1(\R^3)$ norm. In the same spirit, the symbol $\|\cdot\|_{\mathcal{D}^{1,2}}$ will often be used for also for the $\mathcal{D}^{1,2}(\R^3,\R^2)$.

For $s \in \R$ and $w \in H^1(\R^3)$, we define the function
\[%begin{equation}\label{def star}
  (s \star w)(x) := e^{3s/2} w(e^s x).
\]%end{equation}
One can easily check that $\|s \star w\|_{L^2} = \|w\|_{L^2}$ for every $s \in \R$. As a consequence, given $(u,v) \in \cS$, it results that $s \star (u,v)  := (s \star u, s \star v) \in \cS$ as well, for every $s \in \R$.

We consider the real valued function $\Psi_{(u,v)}^\beta(s):= J_\beta(s \star (u,v))$. By changing variables in the integrals, we obtain
\begin{equation}\label{explicit expression}
  \Psi^\beta_{(u,v)}(s)
    = \frac{e^{2s}}{2}\int_{\R^3} |\nabla u|^2 + |\nabla v|^2 - \frac{e^{3s}}{4}\int_{\R^3} \mu_1 u^4  + 2 \be u^2 v^2 + \mu_2 v^4.
\end{equation}
Let us introduce
\[%begin{equation}\label{def E}
  \cE_\be
    := \left\{(u,v) \in \cS:  \int_{\R^3} \mu_1 u^4 + 2 \beta u^2 v^2 + \mu_2 v^4 > 0\right\}.
\]%end{equation}
Clearly $\cE_\be = \cS$ in case $- \sqrt{\mu_1 \mu_2} < \beta < +\infty$, while for $\beta \le - \sqrt{\mu_1 \mu_2}$ it results that $\cE_\be \subset \cS$ is an open subset with strict inclusion. By definition and using \eqref{explicit expression}, it is not difficult to check that:

\begin{lemma}\label{lem: star}
For any $(u,v) \in \cS$, a value $s \in \R$ is a critical point of $\Psi^\beta_{(u,v)}$ if and only if $s \star (u,v) \in \mathcal{P}_\beta$. It results that:
\begin{itemize}
\item[($i$)] If $(u,v) \in \cE_\be$, then there exists a unique critical point $s_{(u,v)}^\beta \in \R$ for $\Psi^\beta_{(u,v)}$, which is a strict maximum point, and is defined by
\begin{equation}\label{def s_u}
  \exp(s_{(u,v)}^\be) = \frac{4 \int_{\R^3} |\nabla u|^2 + |\nabla v|^2}{3 \int_{\R^3} \mu_1 u^4 + 2 \be u^2 v^2 + \mu_2 v^4 }.
\end{equation}
In particular, if $(u,v) \in \mathcal{P}_\beta$, then $s^\beta_{(u,v)}=0$.
\item[($ii$)] If $(u,v) \not \in \cE_\be$, then $\Psi^\beta_{(u,v)}$ has no critical point in $\R$.
\end{itemize}
\end{lemma}

%For the convenience of the reader we state Theorem~2.1 from\cite{BaSo}.
%
%\begin{theorem}\label{thm:A}
%If $\{(\bar u_n,\bar v_n)\}$ is a Palais-Smale sequence for the restriction $J_\be|_{\cP_\be}$, then there exists a possibly different sequence $\{(u_n,v_n)\} \subset \mathcal{C}^\infty_c(\R^3) \cap \cP_\be$ such that
%\begin{itemize}
%\item[($i$)] $\{(u_n,v_n)\}$ is a Palais-Smale sequence for $J_\be$ on $\mathcal{S}$;
%\item[($ii$)] $\|u_n-\bar u_n \|_{H^1(\R^3)} + \|v_n-\bar v_n \|_{H^1(\R^3)} \to 0$ as $n \to \infty$.
%\end{itemize}
%Moreover, if $(u,v)$ is a critical point for $J_\be|_{\cP_\be}$, then $(u,v)$ is a critical point for $J_\be$ on $\mathcal{S}$.
%\end{theorem}

\section{Proof of Theorem \ref{thm:PS}}\label{sec: nat}

For the sake of generality and applications elsewhere, we consider at first the non-symmetric case, i.e. $a_1$ and $a_2$, and also $\mu_1$ and $\mu_2$, are not necessarily equal. In addition, $\be\in\R$ may be arbitrary in this section.

We recall the following definition \cite[Definition 3.1]{Ghou}.

\begin{definition}\label{def hsf}
Let $B$ be a closed subset of a set $X \subset H^1_{\rad}(\R^3,\R^2)$. We say that a class $\mathcal{F}$ of compact subsets of $X$ is a \emph{homotopy stable family with closed boundary $B$} provided
\begin{itemize}
\item[($a$)] every set in $\mathcal{F}$ contains $B$.
\item[($b$)] for any $A \in \mathcal{F}$ and any $\eta \in C([0,1] \times X,X)$ satisfying $\eta(t,x) = x$ for all $(t,x) \in (\{0\} \times X) \cup ([0,1] \times B)$, we have that $\eta(\{1\} \times A) \in \mathcal{F}$.
\end{itemize}
\end{definition}

This section is devoted to the proof of the following statement, and of its equivariant version, Theorem \ref{thm:PSequi}.

\begin{theorem}\label{thm:PS}
Let $\beta \in \R$, and let $\mathcal{F}$ be a homotopy stable family of compact subsets of $\mathcal{P}_\beta$, with closed boundary $B \subset \mathcal{P}_\beta$. Let
\[
c_{\mathcal{F},\beta}:= \inf_{A \in \mathcal{F}} \, \max_{(u,v) \in A} J_\beta(u,v)
\]
Suppose that
\[%begin{equation}\label{eq:hp minmax}
\text{$B$ is contained in a connected component of $\mathcal{P}_\beta$},
\]%end{equation}
and that
\[%begin{equation}\label{eq:hp minmax2}
\max\{\sup J_\beta(B),0\} <  c_{\mathcal{F},\beta} < +\infty.
\]%end{equation}
Then there exists a sequence $\{(u_n,v_n)\}$ with the following properties:
\begin{itemize}
\item[($i$)] $(u_n,v_n) \in \mathcal{P}_\beta$ for every $n$;
\item[($ii$)] $J_\beta(u_n,v_n) \to c_{\mathcal{F},\beta}$ as $n \to \infty$;
\item[($iii$)] $\|\nabla (J_\beta|_{\cS})(u_n,v_n)\| \to 0$ as $n \to \infty$.
\end{itemize}
If moreover we can find a minimizing sequence $\{A_n\}$ for $c_{\mathcal{F},\beta}$ in such a way that $(u,v) \in A_n$ implies $u,v \ge 0$ a.e., then we can find the sequence $\{(u_n,v_n)\}$ satisfying the additional condition
\begin{itemize}
\item[($iv$)] $u_{n}^-, v_{n}^- \to 0$ a.e. in $\R^3$ as $n \to \infty$.
\end{itemize}
\end{theorem}

\begin{remark}
The theorem establishes that, if the assumptions of the minimax principle \cite[Theorem 3.2]{Ghou} are satisfied by the constrained functional $J_\be|_{\mathcal{P}_\be}$, then we can find a ``free" Palais-Smale sequence for $J_\be$ on $\mathcal{S}$, made of elements of $\cP_\beta$. This is exactly what we needed in the proof of \cite[Proposition 3.5]{BaSo} and consequently \cite[Theorem~1.1]{BaSo}. With regard to this, we point out that for the proof of Theorem \ref{thm:PS} the fact that we deal with radial solutions is not needed and never used, and hence the result holds also in the non-radial case. Notice also that the result is true for every $\beta \in \R$, and in particular in this statement we do not assume $\beta<0$.
%We have also a natural counterpart of Theorem 4.1 in \cite{BaSo}. We present such a statement in the appendix, see Theorem \ref{thm:PSeq} below. Using Theorems \ref{thm:PS} and \ref{thm:PSeq} instead of Theorems 2.1 and 4.1 in \cite{BaSo} we can recover Theorems 1.1 and 1.4 in \cite{BaSo}, without modifying the rest of the proofs.
\end{remark}

In order to simplify the notation, from now on we omit the dependence of all the quantities with respect to $\beta$; i.e.\ we write $J$ instead of $J_\be$ etc. For the proof of Theorem \ref{thm:PS}, we define the functional $E:\mathcal{E} \to \R$ by
\[
E(u,v):= J(s_{(u,v)} \star (u,v)).
%= \frac{8 \left( \int_{\R^3} \sum_{i=1}^2 |\nabla u_i|^2\right)^3}{27 \left( \int_{\R^3} \sum_{i,j=1}^2 \beta_{ij} u_i^2 u_j^2\right)^2}.
\]
Using \eqref{def s_u} and the fact that $s_{(u,v)} \star (u,v) \in \mathcal{P}$, it is easy to check that the following equivalent expressions hold:
\begin{equation}\label{expE}
\begin{split}
E(u,v) &  = \frac16 \int_{\R^3}  |\nabla (s_{(u,v)} \star u)|^2 +  |\nabla (s_{(u,v)} \star v)|^2= \frac{e^{2 s_{(u,v)}}   }{6}   \int_{\R^3}  |\nabla u|^2 + |\nabla v|^2  \\
& = \frac{8 \left( \int_{\R^3}  |\nabla u|^2 + |\nabla v|^2 \right)^3}{27 \left( \int_{\R^3}  \mu_1 u^4 + 2\beta u^2 v^2 +  \mu_2 v^4\right)^2}
\end{split}
\end{equation}
for every $(u,v) \in \mathcal{E}$. Moreover, observing that for every $s_1,s_2 \in \R$ and $w \in H^1(\R^3)$
\begin{equation}\label{eq:star}
  (s_1 + s_2) \star w = s_1 \star (s_2 \star w), \quad \text{and} \quad 0 \star w = w,
\end{equation}
it follows immediately by Lemma \ref{lem: star} that $E(s \star (u,v)) = E(u,v)$ for every $s \in \R$ and $(u,v) \in \mathcal{E}$.

%Now we define the functional $R:\mathcal{E} \to \R$,
%\[
%R(\mf{u}):= \frac{8 \left( \int_{\R^3} \sum_{i=1}^2 |\nabla u_i|^2\right)^3}{27 \left( \int_{\R^3} \sum_{i,j=1}^2 \beta_{ij} u_i^2 u_j^2\right)^2}.
%\]
%It is easy to check that $R(\mf{u}) = R(s \star \mf{u})$ for every $s \in \R$. Moreover, by \eqref{def s_u} and using the fact that $s_{\mf{u}} \star \mf{u} \in \mathcal{P}$,
%\begin{align*}
%R(\mf{u}) & = \frac{1}{6} \cdot \left( \frac{4 \int_{\R^3} \sum_{i=1}^2 |\nabla u_i|^2}{3 \int_{\R^3} \sum_{i,j=1}^2 \beta_{ij} u_i^2 u_j^2}\right)^2 \cdot \int_{\R^3} \sum_{i=1}^2 |\nabla u_i|^2 \\
%& =  \frac{e^{2 s_{\mf{u}}}   }{6}   \int_{\R^3} \sum_{i=1}^2 |\nabla u_i|^2 = \frac16 \int_{\R^3} \sum_{i=1}^2 |\nabla (s_{\mf{u}} \star u_i)|^2 = J(s_{\mf{u}} \star \mf{u})
%\end{align*}
%for every $\mf{u} \in \mathcal{E}$.

%Using the explicit expression of $J$ and of $s_{\mf{u}}$, and in particular the fact that $s_\mf{u} \star \mf{u} \in \mathcal{P}$ for any $\mf{u} \in \mathcal{E}$, it is not difficult to check that
%\begin{align*}
%E(\mf{u}) & = \frac{e^{2s_{\mf{u}}}}{2}\int_{\R^3} \sum_{i=1}^2 |\nabla u_i|^2 - \frac{e^{3s_{\mf{u}}}}{4}\int_{\R^3} \sum_{i,j=1}^2 \beta_{ij} u_i^2 u_j^2 \\
%& = \frac16 \int_{\R^3} |\nabla(s_{\mf{u}} \star u_1)|^2 + |\nabla(s_{\mf{u}} \star u_2)|^2 \\
%& = \frac18 \int_{\R^3}   \sum_{i,j=1}^2 \beta_{ij} (s_{\mf{u}} \star u_i)^2 (s_{\mf{u}} \star u_j)^2
%\end{align*}
We aim at proving that a Palais-Smale sequence for $E$ on $\cS$ yields a Palais-Smale sequence for $J$ on $\cS$ with elements on $\mathcal{P}$. In this direction, we need some preliminary lemmas.

\begin{lemma}\label{minP}
There exists $\delta>0$ (depending on $a_1,a_2,\mu_1,\mu_2>0 $ and on $\beta \in \R$) such that
\[
\inf_{(u,v) \in \mathcal{P}} \, \|(u,v)\|_{\mathcal{D}^{1,2}} \ge \delta.
\]
\end{lemma}

\begin{proof}
This is a consequence of the Gagliardo-Nirenberg inequality, which asserts that there exists a universal constant $C>0$ such that
\[
  \int_{\R^3} w^4 \le C \left(\int_{\R^3} w^2\right)^\frac12\left(\int_{\R^3} |\nabla w|^2 \right)^\frac32 \qquad
   \text{for all } w \in H^1(\R^3).
\]
If $(u,v) \in \mathcal{P}$, then $(u,v) \in \mathcal{E}$, and then we have
\begin{align*}
0 <& \left( \int_{\R^3}  |\nabla u|^2 + |\nabla v|^2 \right)^{\frac23}
  = \left( \frac34 \int_{\R^3} \mu_1 u^4 + 2\beta u^2v^2 + \mu_2 v^4 \right)^\frac23   \\
  \le & C  \left(\int_{\R^3} u^4 + v^4\right)^\frac23  \le C \int_{\R^3}  |\nabla u|^2 + |\nabla v|^2
\end{align*}
for a positive constant $C$ depending on the data.
\end{proof}

\begin{lemma}\label{BaSo2.7}
Let $\{u_n\} \subset H^1(\R^3)$, $\{s_n\} \subset \R$, and let us assume that $u_n \to u$ strongly in $H^1(\R^3)$, and $s_n \to s$ in $\R$. Then $s_n \star u_n \to s \star u$ strongly in $H^1(\R^3)$.
\end{lemma}
\begin{proof}
%We have that $s \star u \in H^1(\R^3)$, and
%\[
%\nabla (s \star u)(x) = e^{5/2 s}\nabla u(e^s x) \qquad \text{for a.e. $x \in \R^3$}.
%\]
%Let us show at first that $\|s_n \star u_n - s \star u\|_{L^2(\R^3)} \to 0$. We have
%\begin{multline}\label{11101}
%\|s_n \star u_n - s \star u\|_{L^2(\R^3)}  \le \|
%e^{3/2 s_n} u_n(e^{s_n} \, \cdot ) - e^{3/2 s} u_n(e^{s_n}\,\cdot)\|_{L^2(\R^3)} \\
%+ \| e^{3/2 s} u_n(e^{s_n}\,\cdot)- e^{3/2 s} u(e^{s_n}\,\cdot)\|_{L^2(\R^3)} + \|e^{3/2 s} u(e^{s_n}\,\cdot) - e^{3/2 s} u(e^s\,\cdot)\|_{L^2(\R^3)};
%\end{multline}
%the first term can be estimated directly in the following way:
%\[
%\int_{\R^3} \left( e^{3/2 s_n} u_n(e^{s_n} x ) - e^{3/2 s} u_n(e^{s_n} x)\right)^2\,dx = \left(1-e^{3/2 (s_n-s)}\right)^2 \int_{\R^3} u_n^2 \to 0
%\]
%as $n \to \infty$, since $\|u_n\|_{L^2(\R^3)}$ is bounded and $s_n \to s$. As for the second term, we have
%\[
%\int_{\R^3} \left( e^{3/2 s} u_n(e^{s_n} x ) - e^{3/2 s} u(e^{s_n} x)\right)^2\,dx = e^{3(s-s_n)}\int_{\R^3} \int_{\R^3} \left(u_n-u\right)^2 \to 0
%\]
%as $n \to \infty$, since $u_n \to u$ in $L^2(\R^3)$. Therefore, in order to prove that the right hand side in \eqref{11101} tends to $0$, it remains only to show that $e^{3/2 s} u(e^{s_n}\,\cdot) - e^{3/2 s} u(e^s\,\cdot) \to 0$ in $L^2(\R^3)$.
Let us show at first that $s_n \star u_n \to s \star u$ weakly in $L^2(\R^3)$.
To this purpose, we take any $\varphi \in C^\infty_c(\R^3)$, and for a compact set $K$ containing the support of $\varphi(e^{-s_n}\,\cdot)$ for every $n$ sufficiently large, we observe that
\begin{align*}
\int_{\R^3} e^{3/2 s_n} u_n(e^{s_n} x) \varphi(x)\,dx
&= e^{-3/2s_n} \int_{K} u_n(y) \varphi(e^{-s_n} y)\,dy \\
&\to e^{-3/2 s}\int_{K} u(y) \varphi(e^{-s} y)\,dy = \int_{\R^3} e^{3/2 s}u(e^s x) \varphi(x)\,dx
\end{align*}
as $n \to \infty$, by the dominated convergence theorem. In the same way, we can show that for any $i=1,\dots,3$ and any $\varphi \in C^\infty_c(\R^3)$
\[
\int_{\R^3} \varphi \, \pa_{x_i}(s_n \star u_n)  \to \int_{\R^3} \varphi \, \pa_{x_i}(s \star u),
\]
and as a consequence $s_n \star u_n \to s \star u$ weakly in $H^1(\R^3)$. Furthermore,
\[
\|s_n \star u_n\|_{H^1}^2 = e^{2s_n} \int_{\R^3} |\nabla u_n|^2 + \int_{\R^3} u_n^2  \to e^{2s} \int_{\R^3} |\nabla u|^2 + \int_{\R^3} u^2= \|s \star u\|_{H^1}^2,
\]
and the thesis follows.
%
%
%
%Let us show at first that $s_n \star u_n \to s \star u$ strongly in $L^2(\R^3)$. To this purpose, we take any $\varphi \in C^\infty_c(\R^3)$, and for a compact set $K$ containing the support of $\varphi(e^{-s_n}\,\cdot)$ for every $n$ sufficiently large, we observe that
%\begin{multline*}
%\int_{\R^3} e^{3/2 s_n} u_n(e^{s_n} x) \varphi(x)\,dx = \int_{\R^3} e^{-3/2 s_n} u_n(y) \varphi(e^{-s_n} y)\,dy \\
%= e^{-3/2s_n} \int_{K} u_n(y) \varphi(e^{-s_n} y)\,dy \to e^{-3/2 s}\int_{K} u(y) \varphi(e^{-s} y)\,dy = \int_{\R^3} e^{3/2 s}u(e^s x) \varphi(x)\,dx
%\end{multline*}
%by the dominated convergence theorem. Thus, $s_n \star u_n \to s \star u$ weakly in $L^2(\R^3)$. Furthermore,
%\[
%\|s_n \star u_n\|_{L^2} = \|u_n\|_{L^2} \to \|u\|_{L^2} = \|s \star u\|_{L^2}
%\]
%as $n \to \infty$, and hence $s_n \star u_n \to s \star u$ strongly in $L^2(\R^3)$.
%An analogue argument applies for the convergence of the derivatives, and hence the thesis follows.
\end{proof}

\begin{lemma}\label{lem: tg}
For $(u,v) \in \cS$ and $s\in\R$ the map
\[
  T_{(u,v)} \cS \to T_{s \star (u,v)} \cS,\quad (\varphi_1,\varphi_2) \mapsto s \star (\varphi_1,\varphi_2),
\]
is a linear isomorphism with inverse $(\psi_1,\psi_2) \mapsto (-s) \star (\psi_1,\psi_2)$. In particular, for $(u,v) \in \mathcal{E}$ the map
\[
  T_{(u,v)} \cS \to T_{s_{(u,v)} \star (u,v)} \cS,\quad (\varphi_1,\varphi_2) \mapsto s_{(u,v)} \star (\varphi_1,\varphi_2),
\]
is an isomorphism.
\end{lemma}

\begin{proof}
%By definition, for every $s_1,s_2 \in \R$ and $w \in H^1(\R^3)$
%\[
%(s_1 + s_2) \star w = s_1 \star (s_2 \star w), \quad \text{and} \quad 0 \star w = w.
%\]
%Therefore,
For $(\varphi_1,\varphi_2)\in T_{(u,v)} \cS$ we have
\[
  \int_{\R^3} (s \star u)  (s \star \varphi_1) = \int_{\R^3} e^{3s} \, u(e^{s}x) \varphi_1(e^{s} x) \, dx = \int_{\R^3}  u\varphi_1 = 0.
\]
As a consequence $s \star (\varphi_1,\varphi_2) \in T_{s \star (u,v)} \cS$, and the map is well defined. Clearly it is linear. The result follows easily using \eqref{eq:star}.
\end{proof}

\begin{lemma}\label{lem:su bordo}
Let $\{(u_n,v_n)\} \subset \mathcal{E}$ with $(u_n,v_n) \to (u,v) \in \pa \mathcal{E}$ strongly in $H^1(\R^3,\R^2)$ as $n \to \infty$. Then $E(u_n,v_n) \to +\infty$ as $n \to \infty$.
\end{lemma}

\begin{proof}
If $(u_n,v_n) \to (u,v)$ strongly in $H^1(\R^3,\R^2)$, then it converges also in $L^4(\R^3,\R^2)$, and hence
\[
  0 \le \lim_{n \to \infty}\int_{\R^3}  \mu_1 u_n^4 + 2 \beta u_n^2 v_n^2 +\mu_2 v_n^4 =\int_{\R^3} \mu_1 u^4 + 2 \beta u^2 v^2 +\mu_2 v^4 \le 0,
\]
because $(u,v) \in \pa\mathcal{E}$. On the other hand, since $(u,v) \in \cS$
\[
\lim_{n \to \infty} \int_{\R^3} |\nabla u_{n}|^2 + |\nabla v_{n}|^2  = \int_{\R^3} |\nabla u|^2 + |\nabla v|^2 > 0,
\]
and hence the thesis follows by \eqref{expE}.
\end{proof}

The next lemma is crucial for what follows.

\begin{lemma}\label{lem: diff}
The functional $E$ is of class $C^1$ in $\mathcal{E}$, and
\[
dE(u,v)[ (\varphi_1, \varphi_2)] = dJ(s_{(u,v)} \star (u,v))[s_{(u,v)} \star (\varphi_1, \varphi_2)]
\]
for every $(u,v) \in \mathcal{E}$, for every $(\varphi_1, \varphi_2) \in T_{(u,v)} \cS$.
\end{lemma}
\begin{proof}
By the last expression in \eqref{expE}, it is clear that $E \in C^1(\mathcal{E})$, and using also \eqref{def s_u}
%We observe that, by \eqref{def s_u}, the map $(u,v) \mapsto s_{(u,v)}$ is of class $C^1$. Therefore,
\begin{align*}
dE(u,v)& [(\varphi_1, \varphi_2)]  =  \left(\frac{ 4 \int_{\R^3}  |\nabla u|^2 + |\nabla v|^2}{3 \int_{\R^3}  \mu_1 u^4 + 2\beta u^2 v^2 + \mu_2 v^4} \right)^2 \int_{\R^3}  \nabla u \cdot \nabla \varphi_1 + \nabla v \cdot \nabla \varphi_2
\\
& \qquad - \left(\frac{ 4 \int_{\R^3}  |\nabla u|^2 + |\nabla v|^2}{3 \int_{\R^3} \mu_1 u^4 + 2\beta u^2 v^2 + \mu_2 v^4} \right)^3  \int_{\R^3} \mu_1 u^3 \varphi_1 + \beta  uv (u \varphi_2 + v \varphi_1) + \mu_2 v^3 \varphi_2  \\
& = e^{2 s_{(u,v)}} \int_{\R^3}  \nabla u \cdot \nabla \varphi_1 + \nabla v \cdot \nabla \varphi_2 - e^{3 s_{(u,v)}} \int_{\R^3} \mu_1 u^3 \varphi_1 + \beta  uv (u \varphi_2 + v \varphi_1) +\mu_2 v^3 \varphi_2 \\
& =  \int_{\R^3} \nabla (s_{(u,v)} \star u) \cdot \nabla (s_{(u,v)} \star\varphi_1) +  \nabla (s_{(u,v)} \star v) \cdot \nabla (s_{(u,v)} \star\varphi_2) \\
& \qquad -  \int_{\R^3} \mu_1 (s_{(u,v)} \star u)^3(s_{(u,v)} \star \varphi_1) + \mu_2 (s_{(u,v)} \star v)^3(s_{(u,v)} \star \varphi_2)  \\
& \qquad -  \int_{\R^3} \beta (s_{(u,v)} \star u) (s_{(u,v)} \star v) \big( (s_{(u,v)} \star u) (s_{(u,v)} \star \varphi_2) + (s_{(u,v)} \star v)(s_{(u,v)} \star \varphi_1) \big) \\
& = dJ(s_{(u,v)} \star (u,v))[s_{(u,v)} \star (\varphi_1, \varphi_2)],
\end{align*}
for every $(u,v) \in \mathcal{E}$ and $(\varphi_1, \varphi_2) \in T_{(u,v)}\cS$.
\end{proof}

The immediate corollary of the previous lemmas is that $(u,v) \in \mathcal{E}$ is a critical point for $E$ on $\cS$ if and only if $s_{(u,v)} \star (u,v)$ is a critical point for $J$ on $\cS$, with $s_{(u,v)} \star (u,v) \in \mathcal{P}$. This result is not enough for our purposes, we wish to obtain a similar characterization for Palais-Smale sequences.

\begin{proposition}\label{prop: PS}
Let $\mathcal{G}$ be a homotopy stable family of compact subsets of $\mathcal{E}$ with closed boundary $B$, and let
\[
  e_{\cG} := \inf_{A \in \mathcal{G}} \max_{(u,v) \in A} E(u,v).
\]
Suppose that
\begin{equation}\label{hp minmax}
\text{$B$ is contained in a connected component of $\mathcal{P}$},
\end{equation}
and that
\begin{equation}\label{hp minmax2}
\max\{\sup E(B),0\} < e_{\cG}<+\infty.
\end{equation}
Then there exist two sequences $\{(\tilde u_n,\tilde v_n)\}$ and $\{(u_n,v_n):= s_{(\tilde u_n, \tilde v_n)} \star (\tilde u_n,\tilde v_n)\}$ with the following properties:
\begin{itemize}
\item[($i$)] $\{(\tilde u_n,\tilde v_n)\}$ is a Palais-Smale sequence for $E$ on $\cS$, at level $e_{\cG}$;
\item[($ii$)] $s_{(\tilde u_n,\tilde v_n)} \to 0$ as $n \to \infty$;
\item[($iii$)] $(u_n,v_n) \in \mathcal{P}$ for every $n$;
\item[($iv$)] $\{(u_n,v_n)\}$ is a Palais-Smale sequence for $J$ on $\cS$, at level $e_{\cG}$.
\end{itemize}
If moreover we can find a minimizing sequence $\{D_n\}$ for $e_{\cG}$ in such a way that $(u,v) \in D_n$ implies $u,v \ge 0$ a.e., then we can find the sequence $\{(u_n,v_n)\}$ satisfying the additional condition
\begin{itemize}
\item[($v$)] $u_{n}^-, v_{n}^- \to 0$ a.e. in $\R^3$ as $n \to \infty$.
\end{itemize}
\end{proposition}

\begin{remark}
The proposition says in particular that, if the assumptions of the minimax principle \cite[Theorem 3.2]{Ghou} are satisfied \emph{for the functional $E$}, then we find a Palais-Smale sequence \emph{for the functional $J$}, made of elements in $\mathcal{P}$.
\end{remark}

\begin{proof}
Let $\{D_n\} \subset \mathcal{G}$ be a minimizing sequence for $e_{\cG}$: $\max_{D_n} E \to e_{\cG}$. We define the map
\[
\eta:[0,1] \times \mathcal{E} \to \mathcal{E}, \quad \eta(t,(u,v)) = (t \, s_{(u,v)}) \star  (u,v).
\]
Since $s_{(u,v)}=0$ for any $(u,v) \in \mathcal{P}$ and $B \subset \mathcal{P}$, it is clear that $\eta(t,(u,v))=(u,v)$ for $(t,(u,v)) \in (\{0\} \times \mathcal{E}) \cup ([0,1] \times B)$. Furthermore, by \eqref{def s_u} and Lemma \ref{BaSo2.7}, $\eta$ is continuous. Thus, by Definition \ref{def hsf}
\[
A_n:= \eta(\{1\} \times D_n) = \{s_{(u,v)} \star (u,v): \ (u,v) \in D_n\} \in \mathcal{G}.
\]
Notice that $A_n \subset \mathcal{P}$ for every $n$.

Let $(w,z) \in A_n$. Then $(w,z)=s_{(u,v)} \star (u,v)$ for some $(u,v) \in D_n$, and clearly $E(w,z) = E(s_{(u,v)} \star (u,v)) = E(u,v)$. In particular, $\max_{A_n} E = \max_{D_n} E$, and hence $\{A_n\}$ is another minimizing sequence for $e_{\cG}$, with the property that $A_n \subset \mathcal{P}$ for every $n$. At this point we would like to apply the min-max principle \cite[Theorem 3.2]{Ghou} to this minimizing sequence. A word of caution is needed here, since $\mathcal{E}$ is neither complete, nor connected, and hence in principle the assumptions of \cite[Theorem 3.2]{Ghou} are not satisfied. On the other hand, the connectedness assumption can be avoided considering the restriction of $E$ on the connected component of $\cE$ containing $B$ (by \eqref{hp minmax}, such a connected component does exist; if $B=\emptyset$, we can simply choose a connected component arbitrarily). Regarding the completeness, what is really used in the deformation lemma \cite[Lemma 3.7]{Ghou} is that the sublevel sets $E^d:=\{(u,v) \in \cE: \, E(u,v) \le d\}$ are complete for every $d \in \R$. This follows by Lemma \ref{lem:su bordo}.
Hence, by the min-max principle, \cite[Theorem 3.2]{Ghou}, there exists a Palais-Smale sequence $\{(\tilde u_n,\tilde v_n)\}$ for $E$ on $\mathcal{E}$ at level $e_{\cG}$ with the property that $\dist_{H^1(\R^3,\R^2)}((\tilde u_n,\tilde v_n), A_n) \to 0$ as $n \to \infty$.

Let $s_n:= s_{(\tilde u_n,\tilde v_n)}$. We claim that
\begin{equation}\label{cl111}
s_n \to 0 \quad \text{as $n \to \infty$}.
\end{equation}
In order to prove the claim, we note that since $A_n$ is compact for every $n$, there exists $(w_n,z_n) \in A_n$ with $\dist_{H^1(\R^3,\R^2)}((\tilde u_n,\tilde v_n),A_n) = \|(w_n,z_n)-(\tilde u_n,\tilde v_n)\|_{H^1}$. The function $(w_n,z_n) \in \mathcal{P}$, and hence $s_{(w_n,z_n)} = 0$ for every $n$:
\begin{equation}\label{s=0}
\frac{ 4 \int_{\R^3} |\nabla w_{n}|^2 + |\nabla z_{n}|^2}{ 3 \int_{\R^3} \mu_1 w_n^4 + 2 \beta  w_n^2 z_n^2 + \mu_2 z_n^4 } = 1.
\end{equation}
Notice also that for any $(u,v) \in \cP$ it results
\[
E(u,v) = J(u,v) =  \frac16 \int_{\R^3} |\nabla u|^2+|\nabla v|^2,
\]
so that
\[
\max_{(u,v) \in A_n} E(u,v) \to c \quad \implies \quad \max_{(u,v)\in A_n} \int_{\R^3} |\nabla u|^2+|\nabla v|^2  \le 6c+1
\]
for every $n$ large, and in particular
\begin{equation}\label{11101}
\text{the sequence $\{(w_n,z_n)\}$ is bounded in $H^1$.}
\end{equation}

Now,
\begin{align*}
\int_{\R^3} |\nabla \tilde u_n|^2 &  = \int_{\R^3} |\nabla w_n|^2 +2 \nabla w_n\cdot \nabla (\tilde u_n-w_n) + |\nabla (\tilde u_n-w_n)|^2\\
& = \int_{\R^3} |\nabla w_n|^2 +o(1)
\end{align*}
as $n \to \infty$, since $\|\tilde u_n-w_n\|_{H^1} \to 0$ and $\|w_n\|_{H^1}$ is bounded. Similarly,
\begin{align*}
\int_{\R^3} |\nabla \tilde v_n|^2 = \int_{\R^3} |\nabla z_n|^2 +o(1).
\end{align*}
The fourth order terms can be treated in similar way, using the continuity of the embedding $H^1(\R^3,\R^2) \hookrightarrow L^4(\R^3,\R^2)$:
\begin{align*}
\int_{\R^3} \tilde u_n^4 & = \int_{\R^3} w_n^4 + 4\big(w_n + t(x) (\tilde u_n-w_n) \big)^3 (\tilde u_n-w_n) \\
& = \int_{\R^3}  w_n^4 + o(1),
\end{align*}
where $t(x)\in(0,1)$ comes from the Lagrange theorem; also
\begin{align*}
\int_{\R^3} \tilde v_n^4 = \int_{\R^3} z_n^4 +o(1),
\end{align*}
and finally
\begin{align*}
\int_{\R^3} \tilde u_n^2 \tilde v_n^2 &= \int_{\R^3} w_n^2 z_n^2 + 2 \big(w_n + t(x) (\tilde u_n-w_n) \big) \big(z_n + t(x) (\tilde v_n-z_n) \big)^2 (\tilde u_n-w_n) \\
& \qquad + 2 \int_{\R^3}  \big(w_n + t(x) (\tilde u_n-w_n) \big)^2 \big(z_n + t(x) (\tilde v_n-z_n) \big) (\tilde v_n-z_n) \\
& = \int_{\R^3} w_n^2 z_n^2 + o(1).
\end{align*}
Collecting together the previous estimates, we deduce that
\[
\int_{\R^3} |\nabla \tilde u_{n}|^2 + |\nabla \tilde v_{n}|^2 = \int_{\R^3} |\nabla w_{n}|^2 + |\nabla z_{n}|^2 +r_n
\]
with $r_n \to 0$, and analogously by \eqref{s=0}
\begin{align*}
\int_{\R^3} \mu_1 \tilde u_n^4 + 2 \beta  \tilde u_n^2 \tilde v_n^2 + \mu_2 \tilde v_n^4  & = \int_{\R^3} \mu_1 w_n^4 + 2 \beta  w_n^2 z_n^2 + \mu_2 z_n^4 + o(1) \\
& = \frac43 \left[ \int_{\R^3} \Big( |\nabla w_{n}|^2 + |\nabla z_{n}|^2\Big) +s_n \right]
\end{align*}
with $s_n \to 0$. In conclusion
\[
\frac{4 \int_{\R^3} |\nabla \tilde u_{n}|^2 + |\nabla \tilde v_{n}|^2}{3 \int_{\R^3} \mu_1 \tilde u_n^4 + 2 \beta  \tilde u_n^2 \tilde v_n^2 + \mu_2 \tilde v_n^4} = \frac{  \int_{\R^3} |\nabla w_{n}|^2 + |\nabla z_{n}|^2 + r_n}{  \int_{\R^3} |\nabla w_{n}|^2 + |\nabla z_{n}|^2 +s_n} =1+o(1)
\]
as $n \to \infty$, since the quantity $\int_{\R^3} |\nabla w_{n}|^2 + |\nabla z_{n}|^2$ is bounded from above (by \eqref{11101}) and from below (by Lemma \ref{minP}) by positive values. This proves claim \eqref{cl111}.
In particular, there exist $C_1,C_2>0$ such that
\begin{equation}\label{claim2}
C_1 \le e^{2 s_n} \le C_2 \qquad \text{for every $n$ large}.
\end{equation}

We are now finally ready to conclude. Let $(u_n,v_n):= s_n \star (\tilde u_n,\tilde v_n) \in \mathcal{P}$. We have $E(\tilde u_n,\tilde v_n) = J(u_n,v_n)$ for every $n$, and hence $J(u_n,v_n) \to e_{\cG}$. Furthermore
\begin{align*}
\|dJ(u_n,v_n)\|_*  &= \sup_{ \substack{  (\psi_1,\psi_2) \in T_{(u_n,v_n)} \cS \\ \|(\psi_1,\psi_2)\| = 1}} |dJ(u_n,v_n)[(\psi_1,\psi_2)] | \\
& = \sup_{ \substack{  (\psi_1,\psi_2) \in T_{(u_n,v_n)} \cS \\ \|(\psi_1,\psi_2)\| = 1}} \left|dJ(u_n,v_n)\Big[ s_n \star \big((-s_n) \star  (\psi_1,\psi_2)\big)\Big] \right|,
\end{align*}
where we denoted by $\|\cdot\|_*$ the dual norm in $(T_{(u_n,v_n)} \cS)^*$.
Now, by \eqref{claim2}
\begin{align*}
\underbrace{\min\{C_2^{-1},1\}}_{=: C_3^2}  \|(\psi_1,\psi_2)\|^2 & \le \|(-s_n) \star  (\psi_1,\psi_2)\|^2 \\
& = \sum_i \left(e^{-2s_n} \int_{\R^3}  |\nabla \psi_i|^2 + \int_{\R^3}  \psi_i^2\right) \le  \underbrace{\max\{C_1^{-1},1\}}_{=:C_4^2}   \|(\psi_1,\psi_2)\|^2,
%& \in \left[ \underbrace{\min\{C_4^{-1},1\}}_{=: C_5^2} \|(\psi_1,\psi_2)\|^2, \underbrace{\max\{C_3^{-1},1\}}_{=:C_6^2}   \|(\psi_1,\psi_2)\|^2 \right],
\end{align*}
and recalling also Lemma \ref{lem: tg}, we deduce that
\begin{multline*}
\left\{ (-s_n) \star (\psi_1,\psi_2): \ (\psi_1,\psi_2) \in T_{(u_n,v_n)} \cS, \ \|(\psi_1,\psi_2)\| = 1\right\} \\
\subset \left\{ (\varphi_1,\varphi_2) \in T_{(\tilde u_n,\tilde v_n)} \cS: \ \|(\varphi_1,\varphi_2)\| \in [C_3, C_4]\right\}.
\end{multline*}
The previous argument and Lemma \ref{lem: diff} gives
\begin{align*}
\|dJ(u_n,v_n)\|_* & \le \sup_{ \substack{  (\varphi_1,\varphi_2) \in T_{(\tilde u_n,\tilde v_n)} \cS \\ \|(\varphi_1,\varphi_2)\| \in [C_3, C_4]}} \left|dJ(u_n,v_n)\big[ s_n \star (\varphi_1,\varphi_2)\big] \right|  \\
& =  \sup_{ \substack{  (\varphi_1,\varphi_2) \in T_{(\tilde u_n,\tilde v_n)} \cS \\ \|(\varphi_1,\varphi_2)\| \in [C_3, C_4]}} \left|dJ(s_n \star (\tilde u_n,\tilde v_n))\big[ s_n \star (\varphi_1,\varphi_2)\big] \right| \cdot \frac{\|(\varphi_1,\varphi_2)\|}{\|(\varphi_1,\varphi_2)\|} \\
&   \le C_4 \sup_{ \substack{  (\varphi_1,\varphi_2) \in T_{(\tilde u_n,\tilde v_n)} \cS \\ \|(\varphi_1,\varphi_2)\| \in [C_3, C_4]}} \frac{|dE(\tilde u_n,\tilde v_n)[(\varphi_1,\varphi_2)] |}{\|(\varphi_1,\varphi_2)\|} \to 0
\end{align*}
as $n \to \infty$, as $\{(\tilde u_n,\tilde v_n)\}$ is a Palais-Smale sequence for $E$. To sum up, $\{(u_n,v_n)\}$ is the desired Palais-Smale sequence for $J$ on $\cS$ at level $e_{\cG}$, with $(u_n,v_n) \in \mathcal{P}$ for every $n$.

It remains only to show that, if the original minimizing sequence $\{D_n\}$ is such that $(u,v) \in D_n$ implies $u,v \ge 0$ a.e. in $\R^3$, then $u_{n}^-, v_{n}^- \to 0$ a.e. in $\R^3$. This is a simple consequence of the fact that, under this additional assumptions, we have also $(u,v) \in A_n$ implies $u,v \ge 0$ a.e. in $\R^3$. Thus, by $\dist_{H^1(\R^3,\R^2)}((\tilde u_n,\tilde v_n),A_n) \to 0$, we deduce that $\tilde u_{n}^-, \tilde v_{n}^- \to 0$ a.e. in $\R^3$, and in turn this implies the desired conclusion.
\end{proof}

%\begin{remark}\label{rem: on s_n}
%For future convenience, we stress that under the assumptions of the proposition we can find, in addition to $\{(u_n,v_n)\}$, the Palais-Smale sequence $\{(\tilde u_n,\tilde v_n)\}$ for $E$ on $\cS$ at level $e_{\cG}$, and is such that $(u_n,v_n) = s_n \star (\tilde u_n,\tilde v_n)$ for some $s_n \to 0$.
%\end{remark}

For future convenience, we observe the validity of the following variant of Proposition \ref{prop: PS}, whose proof can be obtained from the previous one with minor modifications.

\begin{proposition}\label{rem: on s_n}
Let $\{(\tilde u_n,\tilde v_n)\}$ be a Palais-Smale sequence for $E$ at level $e \in (0,+\infty)$, and let us suppose that for every $n$ there exists $(w_n,z_n) \in \mathcal{P}$ such that:
\begin{itemize}
\item[($a$)] $\|(\tilde u_n, \tilde v_n) - (w_n,z_n)\|_{H^1} \to 0$ as $n \to \infty$;
\item[($b$)] $\|(w_n,z_n)\|_{\cD^{1,2}} \le C$.
\end{itemize}
Then $s_n:= s_{(\tilde u_n,\tilde v_n)}$ tends to $0$ as $n \to \infty$, and $(u_n,v_n):=s_n \star (\tilde u_n,\tilde v_n)$ satisfies:
\begin{itemize}
\item[($i$)] $(u_n,v_n) \in \mathcal{P}$ for every $n$;
\item[($ii$)] $J(u_n,v_n) \to e$ as $n \to \infty$;
\item[($iii$)] $\|\nabla (J|_{\cS})(u_n,v_n)\| \to 0$ as $n \to \infty$.
\end{itemize}
If moreover $w_n,z_n \ge 0$ a.e., then we have also
\begin{itemize}
\item[($iv$)] $u_{n}^-, v_{n}^- \to 0$ a.e. in $\R^3$ as $n \to \infty$.
\end{itemize}
\end{proposition}

The difference between the above propositions and Theorem \ref{thm:PS} (and \cite[Theorem 2.1]{BaSo}) stays in the fact that here one has to search for minimax structures of $E$ on $\mathcal{E}$, instead of $J$ on $\mathcal{P}$. This difference is only apparent, since by using Proposition \ref{prop: PS} we can easily prove Theorem~\ref{thm:PS}.\\

%With this result in our hands, we can now provide a corrected proof of the following variant of \cite[Theorem 2.1]{BaSo}.

%\begin{theorem}\label{thm:PS}
%Let $\mathcal{F}$ be a homotopy stable family of compact subsets of $\mathcal{P}$ with closed boundary $B \subset \mathcal{P}$. Let
%\[
%c_{\mathcal{F}}:= \inf_{A \in \mathcal{F}} \, \max_{\mf{u} \in A} J(\mf{u}).
%\]
%If $\max\{\sup J(B),0\} < c_{\mathcal{F}}$, then there exists a sequence $\{\mf{u}_n\}$ with the following properties:
%\begin{itemize}
%\item[($i$)] $\mf{u}_n \in \mathcal{P}$ for every $n$;
%\item[($ii$)] $J(\mf{u}_n) \to c_{\mathcal{F}}$ as $n \to \infty$;
%\item[($iii$)] $\|dJ|_{S_{a_1} \times S_{a_2}}(\mf{u}_n)\|_* \to 0$ as $n \to \infty$.
%\end{itemize}
%If moreover we can find a minimizing sequence $\{A_n\}$ for $c_{\mathcal{F}}$ in such a way that $\mf{v}=(v_1,v_2) \in A_n$ implies $v_1,v_2 \ge 0$ a.e., then we can find the sequence $\{\mf{u}_n\}$ satisfying the additional condition
%\begin{itemize}
%\item[($iv$)] $u_{1,n}^-, u_{2,n}^- \to 0$ a.e. in $\R^3$ as $n \to \infty$.
%\end{itemize}
%\end{theorem}
%
%\begin{remark}
%It is important to stress that, according to the definition of homotopy stable family \cite[Definition 3.1]{Gho}, we are assuming that the assumptions of the minimax principle are satisfied for the constrained functional $J|_{\mathcal{P}}$. Nevertheless, from this we can find a Palais-Smale sequence for $J$ on $S_{a_1} \times S_{a_2}$, made of elements in $\mathcal{P}$. This was precisely the content of \cite[Theorem 1.3 and Theorem 2.1]{BaSo}.
%\end{remark}

\begin{altproof}{Theorem~\ref{thm:PS}}
We show that, in the present setting, the assumptions of Proposition \ref{prop: PS} are satisfied, and therefore the thesis follows.

Let $\mathcal{G}$ be the smallest family of compact subsets of $\mathcal{E}$ with closed boundary $B$ (i.e.\ every set in $\mathcal{G}$ contains $B$), which contains $\mathcal{F}$ and is homotopy stable with respect to homotopies $\eta \in C([0,1] \times \mathcal{E},\mathcal{E})$ fixing $(\{0\} \times \mathcal{E}) \cup ([0,1] \times B)$. We check that
\[
  \mathcal{G} = \left\{ \eta(\{1\} \times A) \left| \begin{array}{l} A \in \mathcal{F}, \text{ and }  \text{$\eta \in C([0,1] \times \mathcal{E},\mathcal{E})$ satisfies}\\ \text{$\eta(t,(u,v))=(u,v)$ for $(t,(u,v)) \in (\{0\} \times \mathcal{E}) \cup ([0,1] \times B)$} \\
\end{array}\right.\right\}.
\]
Firstly, since $A \supset B$ for every $A \in \mathcal{F}$, we have that any $D=\eta(\{1\} \times A) \in \mathcal{G}$ contains $\eta(\{1\} \times B) = B$.

Secondly, we have to prove that for every $\eta \in C([0,1] \times \mathcal{E},\mathcal{E})$ fixing $(\{0\} \times \mathcal{E}) \cup ([0,1] \times B)$ and every $D \in \mathcal{G}$, it results that $\eta(\{1\} \times D) \in \mathcal{G}$, i.e. $\eta(\{1\} \times D) = \sigma(\{1\} \times A)$ for some $A \in \mathcal{F}$ and some $\sigma \in C([0,1] \times \mathcal{E},\mathcal{E})$ fixing $(\{0\} \times \mathcal{E}) \cup ([0,1] \times B)$. Since $D \in \mathcal{G}$, there exists $A \in \mathcal{F}$ and $\tau \in C([0,1] \times \mathcal{E},\mathcal{E})$ fixing $(\{0\} \times \mathcal{E}) \cup ([0,1] \times B)$ such that $D=\tau(\{1\} \times A)$. Thus, defining
\[
\sigma(t,(u,v)) = \eta(t, \tau(t,(u,v))),
\]
it follows that $\eta(\{1\} \times D) = \sigma(\{1\} \times A)$, and $\sigma$ is the desired homotopy.

Having checked that $\mathcal{G}$ is a homotopy stable family of compact subsets of $\mathcal{E}$ with closed boundary $B \subset \mathcal{P}$, we consider the associated minimax level
\[
e_\mathcal{G}:= \inf_{D \in \mathcal{G}} \, \max_{(u,v) \in D} E(u,v).
\]

We show that $e_{\mathcal{G}} = c_{\mathcal{F}}$. Recalling that $E(u,v) = J(u,v)$ for $(u,v) \in \mathcal{P}$, this will imply that $ \max\{\sup E(B),0\}<e_{\mathcal{G}} <+\infty$, and permits to apply Proposition \ref{prop: PS}, yielding the thesis of the theorem.

Now, on one side $\mathcal{F} \subset \mathcal{G}$ and $\max_A J = \max_A E$ for every $A \in \mathcal{F}$; therefore,
\begin{equation}\label{cG<cF}
c_\mathcal{F} =  \inf_{A \in \mathcal{F}} \, \max_A J = \inf_{A \in \mathcal{F}} \, \max_A E  \ge e_{\mathcal{G}}.
\end{equation}

In the opposite direction, we prove that for every $\eps>0$ there exists $A \in \mathcal{F}$ such that $\max_{A} J < e_{\mathcal{G}} +\eps$. This implies that $c_{\mathcal{F}} \le e_{\mathcal{G}}$, and, together with \eqref{cG<cF}, completes the proof. For $\eps>0$, let $D \in \mathcal{G}$ with $\max_D E < e_{\mathcal{G}} +\eps$. By definition of $\cG$, it results $D=\eta(\{1\} \times A')$ for some $\eta \in C([0,1] \times \mathcal{E},\mathcal{E})$ fixing $(\{0\} \times \mathcal{E}) \cup ([0,1] \times B)$ and some $A' \in \mathcal{F}$. Let us consider
\[
\tau:[0,1] \times \mathcal{E} \to \mathcal{E}, \qquad \tau(t,(u,v)) = (t s_{(u,v)}) \star(u,v),
\]
and
\[
\sigma:[0,1] \times \mathcal{P} \to \mathcal{P}, \qquad \sigma(t,(u,v)):= \tau(1,\eta(t,(u,v))) = s_{\eta(t,(u,v))} \star \eta(t,(u,v)).
\]
By Lemma \ref{BaSo2.7} and \eqref{def s_u}, it is not difficult to check that $\sigma \in C([0,1] \times \mathcal{P},\mathcal{P})$, and clearly $\sigma$ fixes $(\{0\} \times \mathcal{P}) \cup ([0,1] \times B)$. But then $A:=\sigma(\{1\} \times A') \in \mathcal{F}$ by definition of homotopy stable family. The crucial observation is that $A = \tau(\{1\} \times D)$. %Thus, $\tau(\{1\} \times D) \in \mathcal{F}$ and
%%, and for every $(u,v) \in D$ there exists $\mf{v} \in \sigma(\{1\} \times A)$ with $\tau(1,(u,v)) = \sigma(1,\mf{v})$. It follows in particular that
%$J(\tau(\{1\} \times D)) = J(\sigma(\{1\} \times A))$.
This is important since by definition
\[
E(u,v) = J(s_{(u,v)} \star (u,v)) = J(\tau(1, (u,v)))
\]
for every $(u,v) \in D$, which implies that $E(D) = J(\tau (\{1\} \times D)) = J(A)$. In particular, this gives $\max_A J = \max_D E < e_{\cG} +\eps$, and, since $A \in \mathcal{F}$, completes the proof.
\end{altproof}

To conclude this section, we observe that for the proof of Theorem \ref{thm:PSequi}, which is the equivariant version of Theorem \ref{thm:PS}, we can simply use an equivariant minimax theorem (see e.g.\ \cite[Theorem 7.2]{Ghou}) instead of the classical version \cite[Theorem 3.2]{Ghou}. The rest of the argument remains untouched.

\section{A partial Palais-Smale condition}\label{sec: PS cond}

From now on we focus on the symmetric system \eqref{system} with $a_1=a_2=a$ and $\mu_1=\mu_2=\mu$. Without loss of generality we fix $\mu=1$; this choice simplifies some expressions.

%Given $a>0$, let us introduce the scalar functional $I:S_a^r \to \R$ defined by
%\[
%I(w):= \int_{\R^3} \frac12|\nabla u|^2 - \frac14 w^4.
%\]
%Critical points of $I$ on $S_a^r$ are solutions to \eqref{syst scalar} for some $\lambda \in \R$. It is well known that
We recall that problem \eqref{syst scalar} has a unique positive radial solution $w_0$ (for a suitable $\lambda<0$). We denote by $\ell$ the energy level associated to $w_0$, that is
\[%begin{equation}\label{ell}
  \ell:= I(w_0), \quad I(w):= \int_{\R^3} \frac12|\nabla  w|^2 - \frac14 w^4.
\]%end{equation}
It is well known that $\ell>0$ is the ground state energy level of \eqref{syst scalar} (see Section 2 in \cite{BaJeSo} for a complete discussion).

We wish to investigate the behavior of any Palais-Smale sequence for the constrained functional $J_\beta|_{\cP_\be}$. We start with a preliminary remark.

\begin{lemma}\label{lem: basic J}
  The constrained functional $J_\be|_{\cP_\be}$ is bounded from below and coercive.
\end{lemma}

\begin{proof}
The statement follows straightforwardly from the fact that
\begin{equation}\label{J on P}
  J_\be(u,v) = \frac16 \int_{\R^3} |\nabla u|^2 + |\nabla v|^2 = \frac18\int_{\R^3}  u^4  + 2\beta u^2 v^2 +  v^4
\end{equation}
for any $(u,v) \in \cP_\be$.
\end{proof}

The lemma implies that, if we have a Palais-Smale sequence $\{(u_n,v_n)\}$ for $J_\be$ at a finite level, and $(u_n,v_n) \in \mathcal{P}_\beta$ for every $n$, then $\{(u_n,v_n)\}$ is bounded.
%Notice that, as a consequence of this and of Theorem~\ref{thm:PS}, any Palais-Smale sequence for $J_\be|_{\cP_\be}$ gives a \emph{bounded} Palais-Smale sequence for $J_\be|_{\cS}$ at a finite level.
The existence of bounded Palais-Smale sequences for problems with $L^2$-constraints is a highly non-trivial fact, hence working on the constraint $\cP_\be$ is extremely helpful.

\begin{proposition}\label{prop: PS cond}
Let $\beta \le 0$ be fixed. Let $\{(u_n,v_n)\}$ be a Palais-Smale sequence for $J_\be|_\cS$ at level $c \in (0,+\infty)$, with
\[
  u_n^-,v_n^- \to 0 \quad \text{a.e. in $\R^3$}, \quad \text{and} \quad  (u_n,v_n) \in \mathcal{P}_\beta.
\]

a) If $c \neq \ell$, then up to a subsequence $(u_n,v_n) \to (u,v)$ strongly in $H^1(\R^3,\R^2)$, and $(u,v)$ is a solution to \eqref{system} for some $\lambda_1,\lambda_2<0$. \\

b) If $c=\ell$, then one of the following alternatives occurs:
\begin{itemize}
\item[($i$)] $(u_n,v_n) \to (u,v)$ strongly in $H^1(\R^3,\R^2)$ up to a subsequence, where $(u,v)$ is a solution to \eqref{system} for some $\lambda_1,\lambda_2<0$ with $J_\be(u,v) = \ell$.
\item[($ii$)] either $u_n \to w_0$ strongly in $H^1(\R^3)$ and $v_n \to 0$ strongly in $\mathcal{D}^{1,2}(\R^3)$, or $v_n \to w_0$ strongly in $H^1(\R^3)$ and $u_n \to 0$ strongly in $\mathcal{D}^{1,2}(\R^3)$, up to a subsequence.
\end{itemize}
\end{proposition}

\begin{proof}
We refine the analysis from \cite{BaJeSo, BaSo} to prove the convergence of the Palais-Smale sequences. The phrase ``up to a subsequence" will be implicitly understood in this proof.
%We slightly change the notation with respect to the statement. Let $\{(\bar u_n,\bar v_n)\}$ be a Palais-Smale sequence for $J_\be|_{\cP_\be}$, and let $\{(u_n,v_n)\}$ be the associated Palais-Smale sequence for $J_\be$ on $\cS$, given by Theorem~\ref{thm:A}. In what follows we investigate the convergence of $\{(u_n,v_n)\}$; by point ($ii$) in Theorem~\ref{thm:A}, Proposition \ref{prop: PS cond} will follow.

The weak convergence of $(u_n,v_n)$ to a limit $(\bar u, \bar v) \in H^1(\R^3,\R^2)$ follows directly from Lemma \ref{lem: basic J}. By compactness of the embedding $H^1_{\rad}(\R^3) \hookrightarrow L^4(\R^3)$, $(u_n,v_n) \to(\bar u, \bar v)$ strongly in $L^4$ and a.e.\ in $\R^3$. Notice also that by \eqref{J on P}
\begin{equation}\label{sequence bdd below}
  \int_{\R^3} |\nabla u_n|^2 + |\nabla v_n|^2 \ge 5 c >0
\end{equation}
for every $n$ sufficiently large. Now, since $dJ_\be |_{\cS}(u_n,v_n) \to 0$ (and using the fact that the problem is invariant under rotation), by the Lagrange multipliers rule
there exist two sequences of real numbers $(\lambda_{1,n})$ and $(\lambda_{2,n})$ such that
\begin{multline}\label{gradient to zero}
\int_{\R^3} \left( \nabla u_n \cdot  \nabla \varphi + \nabla v_n \cdot \nabla \psi -  u_n^3 \varphi - v_n^3 \psi - \beta u_n v_n(u_n \psi + v_n \varphi) \right) \\
-\int_{\R^3} \left(\lambda_{1,n} u_n \varphi +\lambda_{2,n} v_n \psi\right) = o(1) \|(\varphi,\psi)\|_{H^1}
\end{multline}
for every $(\varphi, \psi) \in H^1(\R^3,\R^2)$, with $o(1) \to 0$ as $n \to \infty$. Using the boundedness of $\{(u_n,v_n)\}$ and equation \eqref{sequence bdd below}, one can prove as in \cite[Lemma 3.8]{BaJeSo} that $\lambda_{1,n} \to \lambda_1$ and $\lambda_{2,n} \to \lambda_2$, and $\la_1+\la_2<0$, hence at least one of $\lambda_1$ and $\lambda_2$ is a strictly negative value. If $\lambda_1 <0$ (resp. $\lambda_2<0$), then $u_n \to \bar u$ (resp. $v_n \to \bar v$) strongly in $H^1(\R^3)$ by \cite[Lemma 3.9]{BaJeSo}. Notice also that, by weak and a.e.\ convergence and by \eqref{gradient to zero}, the limit $(\bar u, \bar v) \in H^1(\R^3,\R^2)$ solves
\begin{equation}\label{system+}
\begin{cases}
-\Delta \bar u - \lambda_1 \bar u =   \bar{u}^3+ \beta \bar u \bar{v}^2 & \text{in $\R^3$} \\
-\Delta \bar v- \lambda_2 \bar v =  \bar{v}^3 +\beta \bar{u}^2 \bar v & \text{in $\R^3$}\\
\bar u \ge 0, \ \bar v \ge 0 & \text{in $\R^3$},
\end{cases}\qquad \text{for some $\lambda_1,\lambda_2 \in \R$}.
\end{equation}

So far we showed that, independently of the level $c$, the Palais-Smale sequence $\{(\bar u_n,\bar v_n)\}$ tends weakly to a solution of \eqref{system+} (notice that the mass constraint is not present), one of the Lagrange multipliers $\lambda_i$ is negative, and the corresponding component is strongly convergent. Without loss of generality, we can suppose that $\lambda_1<0$, so that $u_n \to \bar u$ strongly. If $\lambda_2<0$, then also $v_n \to \bar v$ strongly in $H^1(\R^3)$ (see \cite[Lemma 3.9]{BaJeSo} again). In what follows we prove that if $c \neq \ell$, then it is necessary that $\lambda_2<0$; while if $c = \ell$, then it is possible that $\lambda_2 \ge 0$, but in such a situation $v_n \to 0$ strongly in $\mathcal{D}^{1,2}(\R^3)$, and $u_n \to w_0$ strongly in $H^1(\R^3)$.

Suppose then that $\lambda_2 \ge 0$. Since $\lambda_1<0$, the function $\bar u$ decays exponentially at infinity, see \cite[Lemma 3.11]{BaSo}. As a consequence, if $\lambda_2 \ge 0$, then
\[
-\Delta \bar v + c(x) \bar v \ge 0 \quad \text{in $\R^3$}, \quad \text{with} \quad 0 \le c(x):= - \beta \bar{u}^2(x) \le C e^{-C |x|}.
\]
Since moreover $\bar v \ge 0$ in $\R^3$ and $\bar v \in H^1(\R^3)$, by the Liouville-type Lemma 3.12 in \cite{BaSo} we infer that $\bar v \equiv 0$ in $\R^3$. But then $\bar u$ is positive and solves \eqref{syst scalar}, and by uniqueness $\bar u= w_0$. It is well known that any radial solution to \eqref{syst scalar} stays in
\[
\mathcal{M}:= \left\{ u \in S_a^r: \int_{\R^3} |\nabla u|^2 = \frac34 \int_{\R^3} u^4\right\},
\]
and hence
\[
\ell = I(w_0) = \frac{1}{8} \int_{\R^3} w_0^4.
\]
Consequently, using that $(u_n,v_n) \to (w_0,0)$ in $L^4$, and recalling \eqref{J on P}, we have
%\begin{equation}\label{eq37}
\[
\begin{split}
c &= \lim_{n \to \infty} J_\be(u_n,v_n)
   = \lim_{n\to\infty} \frac{1}{8}   \int_{\R^3} u_n^4 + \beta u_n^2 v_n^2 + v_n^4  = \frac{1}{8} \int_{\R^3} w_0^4 = \ell.
\end{split}
\]
Therefore, in the case $c \neq \ell$ we must have $\lambda_2<0$, and hence $(u_n,v_n) \to (\bar u, \bar v)$ strongly in $H^1$. If on the other hand $c = \ell$, we proved that in case $\lambda_2 \ge 0$ (that is, if we do not have strong convergence of the whole Palais-Smale sequence) we have $u_n \to w_0$ strongly and $v_n \wc 0$ weakly in $H^1$. To check that indeed $v_n \to 0$ strongly in $\mathcal{D}^{1,2}$, it is sufficient to recall that $(u_n,v_n) \in \cP_\be$ for every $n$, so that
\[
\int_{\R^3} |\nabla v_n|^2 = \frac34\int_{\R^3} \left( u_n^4 + 2\beta u_n^2 v_n^2+ v_n^4\right) - \int_{\R^3} |\nabla u_n|^2 \to \frac34 \int_{\R^3} w_0^4 - \int_{\R^3} |\nabla w_0|^2
\]
as $n \to \infty$. Since $w_0 \in \mathcal{M}$, the last term is equal to $0$, that is $\|v_n\|_{\mathcal{D}^{1,2}} \to 0$, as desired.
\end{proof}

\section{Dependence of the energy level $\inf_{\cP_\be} J_\be$ with respect to $\beta$}\label{sec: energy}

In the first part of this section we analyze the behavior of the energy level $\inf_{\mathcal{P}_\beta} J_\beta$ when $\beta$ varies. We stress that $\beta=0$ is included in our analysis. We set
\[
  m_{\beta}:= \inf_{\mathcal{P}_\beta} J_\beta,
\]
and recall the explicit expression of the functional $E_\beta(u,v) = J_\beta(s^\beta_{(u,v)} \star (u,v))$, see \eqref{expE}.

\begin{lemma}\label{inf ray}
We have
\[
m_\beta= \inf_{(u,v) \in \cE_\be} E_\beta(u,v).
\]
\end{lemma}

\begin{proof}
%We can argue as in \cite[Conclusion of the proof of Theorem 1.2]{BaJeSo}.
For every $(u,v) \in \mathcal{P}_\beta$, the value $s^\beta_{(u,v)}$ defined by \eqref{def s_u} is equal to $0$, and hence by definition of $E_\beta$
\begin{align*}
J_\beta(u,v) &= E_\beta(u,v) \ge \inf_{\cE_\be} E_\beta \quad \implies \quad m_\beta \ge \inf_{\cE_\be} E_\beta.
\end{align*}
%This shows that $m_\beta \ge \inf_{\cE_\be} E_\beta$.
For the reverse inequality, we note that for any $(u,v) \in \cE_\be$
% recall that $R_\beta(u,v) = R_\beta(s \star (u,v))$ for every $s \in \R$. Therefore, by Lemma \ref{lem: star} we deduce that for any $(u,v) \in \cE_\be$ we have
\[
  E_\beta(u,v) =  J_\beta(s^\beta_{(u,v)} \star (u,v)) \ge m_\beta \quad \implies \quad \inf_{\cE_\be} E_\beta \ge m_\beta. \qedhere
\]
\end{proof}

\begin{lemma}\label{lem: monot}
The level $m_\beta$ is monotone non-increasing in $\beta$. In particular, $m_\beta \ge m_0>0$ for every $\beta \le 0$.
\end{lemma}

\begin{proof}
Suppose that $\beta_1<\beta_2 \le 0$ but $m_{\beta_2} > m_{\beta_1}$. Notice that
\begin{equation}\label{13gen1}
  \int_{\R^3} u^4 + v^4 + 2\beta_1 u^2 v^2 \le \int_{\R^3} u^4 + v^4 + 2\beta_2 u^2 v^2,
\end{equation}
and hence $\cE_{\beta_1} \subset \cE_{\beta_2}$. Since $m_{\beta_2} > m_{\beta_1}$, there exists $(u,v) \in \cE_{\beta_1}$ such that $m_{\beta_2} > E_{\beta_1}(u,v) \ge m_{\beta_1}$. But then $(u,v) \in \cE_{\beta_2}$, and using again \eqref{13gen1} we deduce that
\[
  m_{\beta_2} > E_{\beta_1}(u,v) \ge E_{\beta_2}(u,v) \ge \inf_{\cE_{\beta_2}} E_{\beta_2} = m_{\beta_2},
\]
a contradiction.

The inequality $m_0>0$ is the content of Lemma \ref{minP} for $\beta=0$.
\end{proof}

We can now show that, in fact, $m_\beta$ takes always the same value, independently of $\beta \le 0$.

\begin{lemma}\label{lem: inf P}
For every $\beta \le 0$ it results that $m_\beta =\ell$.
\end{lemma}

\begin{proof}
We show first that $m_\beta \ge \ell$, and to this purpose it is sufficient to show that $m_0 \ge \ell$, due to Lemma \ref{lem: monot}. Arguing by contradiction we suppose that $0 < m_0 < \ell$, and consider a minimizing sequence $\{(u_n,v_n)\}$ for $m_0$. It is not restrictive to assume that $u_n,v_n \ge 0$ a.e in $\R^3$, and hence by Proposition \ref{prop: PS cond} we have that up to a subsequence $(u_n,v_n) \to (u_0,v_0)$ strongly in $H^1$. Since $\beta = 0$, the system \eqref{system} is given by two uncoupled equations, and both $u_0$ and $v_0$ are positive radial solutions to \eqref{syst scalar}. By uniqueness, we deduce that $u_0 = v_0 = w_0$, and hence
\[
  \ell > m_0 = J_0(u_0,v_0) = I(u_0) + I(v_0) = 2 \ell \quad \text{with} \quad \ell >0,
\]
a contradiction.

Now we show that $m_\beta \le \ell$. According to Lemma 3.10 in \cite{BaSo}, there exists a sequence $\{(u_n,v_n)\} \subset \cP_\be$ such that $J_\be(u_n,v_n) \to \ell$, $u_n \to w_0$ strongly in $H^1$, and $v_n \to 0$ strongly in $\mathcal{D}^{1,2}$. The inequality $m_\be \le \ell$ follows immediately.
\end{proof}

Let us consider the involution $\sigma: H^1_{\rad}(\R^3,\R^2) \to H^1_{\rad}(\R^3,\R^2)$, $\sigma(u,v) = (v,u)$. Notice that both $J_\be$ and $\cP_\be$ are $\sigma$-invariant, by the symmetry of \eqref{system}.
%$J_\be(\sigma(u,v)) = J_\be(u,v)$ and $\cP_\be$ is $\sigma$-invariant (that is $\sigma(\cP_\be) = \cP_\be$).
Moreover $\si$ has no fixed points in $\cP_\be$ for $\beta \le -1$, because $(u,u) \in \cP_\be$ implies 
\[
  0 < \int_{\R^3} |\nabla u|^2 = \frac34 (1+\beta) \int_{\R^3} u^4.
\]

Next we consider the fixed point set $\cP_\be^\si := \{(u,v)\in\cP_\beta: u=v\}$ as $\be\to-1$, and the infimum $m_\be^\si := \inf_{\cP_\be^\si} J_\be$.
\begin{lemma}\label{lem:mbeta}
For $\be \downarrow -1$ there holds $m_\be^\si \to +\infty$.
\end{lemma}

\begin{proof}
This is a simple consequence of the Gagliardo -Nirenberg inequality: if $(u,u) \in \mathcal{P}_\beta^\sigma$, then
\[
\int_{\R^3} |\nabla u|^2 = \frac34 (1+\beta) \int_{\R^3} u^4 \le \frac{3 C a}{4} (1+\beta)\left( \int_{\R^3} |\nabla u|^2\right)^\frac32,
\]
whence it follows that
\[
 J_\beta(u,u)   =  \frac13\int_{\R^3} |\nabla u|^2  \ge \frac{16}{27C^2a^2(1+\beta)^2}. \qedhere
\]
\end{proof}

\section{The minimax scheme}\label{sec: many PS}

In this section we set up a minimax scheme using the Krasnoselskii genus-type argument. For any closed $\sigma$-invariant set $A \subset \cE_\be$, we define the \emph{genus} $\gamma(A)$ as the smallest integer $n \in \N \cup \{0\}$ such that there exists a continuous map $h:A \to \R^n \setminus \{0\}$ with $h(\sigma(u,v)) = - h(u,v)$ for every $(u,v) \in A$. If no such map exists we set $\gamma(A) = +\infty$. We also note that $\gamma(\emptyset) = 0$. Below, we report some standard properties of $\gamma$, stated and proved for instance in \cite[Lemma 4.4]{DaWeWe}.

\begin{lemma}\label{lem: prop gamma}
Let $A,B \subset \cE_\be$ be closed and $\sigma$-invariant. We have:
\begin{itemize}
\item[($i$)] if $A \subset B$, then $\gamma(A) \le \gamma(B)$.
\item[($ii$)] $\gamma(A \cup B) \le \gamma(A) + \gamma(B)$.
\item[($iii$)] If $h:A \to \cE_\be$ is continuous and $\sigma$-equivariant, i.e.\ $h(\sigma(u,v)) = \sigma(h(u,v))$ for every $(u,v) \in A$, then $\gamma(A) \le \gamma(\overline{h(A)})$.
\end{itemize}
For a subset $A\subset \cE_\be$ that does not contain fixed points of $\sigma$, there holds:
\begin{itemize}
\item[($iv$)] if $\gamma(A)>1$, then $A$ is an infinite set.
\item[($v$)] If $A$ is compact, then $\gamma(A) <+\infty$, and there exists a relatively open $\sigma$-invariant neighborhood $N$ of $A$ in $\cE_\be$ such that $\gamma(A) = \gamma(\overline{N})$.
\end{itemize}
Finally,
\begin{itemize}
\item[($vi$)] if $S$ is the boundary of a bounded symmetric neighborhood of zero in a $k$-dimensional normed vector space and $\psi:S \to \cE_\be$ is a continuous map satisfying $\psi(-s) = \sigma(\psi(s))$, then $\gamma(\psi(S)) \ge k$.
\end{itemize}
\end{lemma}

Let now $\mathcal{A}_\beta :=\{A \subset \cP_\be: \ \text{$A$ is closed and $\sigma$-invariant}\}$, and, for any $k \in \N$, let us define
\[
  \mathcal{A}_{k,\beta}:=\{ A \in \mathcal{A}_\beta: \ \text{$A$ is compact and $\gamma(A) \ge k$}\}.
\]
We define the minimax level
\[
  c_k = c_{k,\beta} := \inf_{A \in \mathcal{A}_{k,\beta}} \max_{(u,v) \in A} J_\beta(u,v).
    %= \inf\{c>0: \ga(J_\be^c)\ge k\}
\]
%where $J_\be^c = \{(u,v)\in \cP_\be: J_\be(u,v)\le c\}$.

\begin{lemma}\label{lem: c_k finite}
Any $c_{k,\beta}$ is a real number, that is $\mathcal{A}_{k,\beta} \neq \emptyset$ for every $k$. Moreover, for every $k \ge 1$ there exists $C_k>0$ independent by $\beta<0$ such that $\ell \le c_{k,\beta} < C_k$.
\end{lemma}

\begin{proof}
We choose a $k$-dimensional subspace $W$ of $\{w \in H^1_{\rad}(\R^3): \int_{\R^3} w = 0\}$, and set $T:=\{ w \in W: \|w\|_{H^1} = 1\}$. Let us consider the maps $\phi:T \to \mathcal{S}$  and $\psi: T \to \cP_\be$ defined by
\[
  \phi(w):= \left( \frac{a}{\|w^+\|_{L^2}} w^+, \frac{a}{\|w^-\|_{L^2}} w^-\right), \quad \text{and} \quad
  \psi(w):=  s^\beta_{\phi(w)} \star \phi(w),
\]
where $s^\beta_{\phi(w)}$ is given by \eqref{def s_u}. Notice that, since the components of $\phi(w)$ have disjoint positivity sets, $s^{\beta}_{\phi(w)}$ is independent of $\beta$; hence, $\psi(T) \subset \cP_\be$ for every $\beta$. Analogously, any function in $\psi(T)$ has components with disjoint positivity set, so that $J_\be(u,v)$ is independent of $\beta$ for every $(u,v) \in \psi(T)$.

Observe that $\psi$ is continuous by Lemma \ref{BaSo2.7} and that $\psi(-w) = \sigma(\psi(w))$. Now Lemma~\ref{lem: prop gamma} ($vi$) implies $\gamma(\psi(T)) \ge k$, $\psi(T) \subset \mathcal{P}_\beta$ is $\sigma$-invariant, and $\psi(T)$ is compact as the continuous image of a compact set. Consequently $\psi(T) \in \mathcal{A}_{k,\beta}$ for every $\beta$, and moreover
\[
  c_{k,\beta} \le \max_{\psi(T)} J_\beta =: C_k,
\]
as desired. The fact that $c_{k,\beta} \ge \ell $ for every $k \in \N$ and $\beta<0$ follows simply by the fact that
\[
\ell= \inf_{\mathcal{P}_\beta} J_\beta = c_{1,\beta} \le c_{k,\beta},
\]
by Lemma \ref{lem: inf P}.
\end{proof}

Now we define for $k\in\N$:
\[
  \be_k := \inf\{\be\in(-1,0): c_{k+1,\beta} \ge m_\be^\si\}.
\]
As a consequence of Lemmas~\ref{lem:mbeta} and \ref{lem: c_k finite} we have $\be_k\in (-1,0)$. Clearly $c_{k,\beta} < m_\be^\si$ provided $\be < \be_k$.

\begin{lemma}\label{ex PS seq}
For any $k\in\N$ and any $\be<\be_k$ there exists a Palais-Smale sequence $\{(u_n^k,v_n^k)\}$ for $J_\be$ on $\cS$ at level $c_{k,\be}$, satisfying the additional conditions $(u_n^k)^-,(v_n^k)^- \to 0$ a.e. in $\R^3$ as $n \to \infty$, and $\{(u_n^k,v_n^k)\} \subset \cP_\be$.
\end{lemma}

\begin{proof}
%We aim at applying Theorem~7.2 in \cite{Ghou} (the equivariant strong form of the minimax principle). To this purpose,
Using point ($iii$) in Lemma \ref{lem: prop gamma}, it is immediate to check that the family $\mathcal{A}_{k,\be}$ is a $\sigma$-homotopy stable family of compact subsets of $\cP_\be$ with boundary $\emptyset$, according to Definition \ref{def hsf equi}. Moreover, by Lemma \ref{lem: c_k finite} assumption \eqref{hp minmax2 equi} is satisfied.
%
%
%the superlevel set
%\begin{equation}\label{def super}
%  J_{\be, c_{k,\be}}:= \left\{(u,v) \in \cP_\be: J_\be(u,v) \ge c_{k,\be} \right\}
%\end{equation}
%is $\sigma$-invariant and \emph{dual} to $\mathcal{A}_{k,\be}$, that is $J_{\be, c_{k,\be}} \cap A \neq \emptyset$ for every $A \in \mathcal{A}_{k,\be}$.
Let then $\{A_n\} \subset \mathcal{A}_{k,\be}$ be a minimizing sequence for $c_{k,\be}$: $\max_{A_n} J_\be \to c_{k,\be}$ as $n\to\infty$. We note that then also $|A_n|$ is a minimizing sequence, where
\begin{equation}\label{mod set}
  |A_n|:= \{(|u|, |v|): \ (u,v) \in A_n\}, \qquad \text{for all }n.
\end{equation}
Indeed, $|A_n|$ inherits the equivariancy and the compactness from $A_n$, and by point ($iii$) in Lemma \ref{lem: prop gamma} we have $\gamma(|A_n|) \ge  \gamma(A_n) \ge k$ for every $k$.
As a consequence, the thesis follows directly from Theorem \ref{thm:PSequi}.
%Theorem 7.2 in \cite{Ghou} is applicable, and implies the existence of a Palais-Smale sequence $\{(u_n,v_n)\}$ for $J_\be|_{\cP_\be}$ at level $c_k$ with the properties that
%\[
%  \dist_{H^1(\R^3,\R^2)}((u_n,v_n), |A_n|) \to 0 \quad \text{and} \quad
%  \dist_{H^1(\R^3,\R^2)}((u_n,v_n), J_{\be, c_{k}} ) \to 0.
%\]
%In particular, $u_n^-, v_n^- \to 0$ a.e.\ in $\R^3$.
\end{proof}

Now we aim at showing the validity of a multiplicity result of Lusternik-Schnirelman type. We define the critical set
\[
  \mathcal{K}_c^+ := \left\{(u,v) \in \cS: \ u,v \ge 0 \text{ a.e. in $\R^3$}, \ J_\be(u,v)= c, \
                         d J_\be|_{\cS}(u,v) = 0\right\}.
\]
By the Pohozaev identity $\mathcal{K}_c^+ \subset \cP_\be$, and clearly $\mathcal{K}_c^+$ is $\sigma$-invariant.

\begin{lemma}\label{lem: multiplicity}
Fix $\be<\beta_{k+p}$ and suppose that $c= c_{j,\be} = c_{j+1,\be} = \dots= c_{j+p,\be}$ for some $j \ge 1$, $p \ge 0$. If $c \neq \ell$, then $\gamma(\mathcal{K}_c^+) > p$.
\end{lemma}

For the proof, we introduce new minimax classes as follows: we consider $\mathcal{B}_\beta :=\{A \subset \cE_\be: \ \text{$A$ is closed and $\sigma$-invariant}\}$, and, for any $k \in \N$,
\[
  \mathcal{B}_{k,\beta}:=\{ A \in \mathcal{B}_\beta: \ \text{$A$ is compact and $\gamma(A) \ge k$}\}.
\]
The associated minimax levels are
\[
  e_k = e_{k,\beta} := \inf_{A \in \mathcal{B}_{k,\beta}} \max_{(u,v) \in A} E_\beta(u,v).
\]

\begin{lemma}\label{lem: e_j}
It results that $e_{k,\be} = c_{k,\be}$, for every $k \in \N$.
\end{lemma}
\begin{proof}
Let $D \subset \cB_{k,\be}$ such that $\max_D E_\be < e_{k,\be} +\eps$, and let us consider the map
\[
h(u,v) = s_{(u,v)}^\be \star (u,v).
\]
By \eqref{def s_u} and Lemma \ref{BaSo2.7}, it is not difficult to check that $h$ is continuous and $\sigma$-equivariant. Then, by Lemma \ref{lem: prop gamma}-($iii$), the compact $\sigma$-invariant set $A:= h(D)$ satisfies $\gamma(A) \ge k$, and hence $A \in \cA_{k,\be}$. By definition, $E_\be(u,v) = J_\be(h(u,v))$ for any $(u,v) \in \cE_\be$, and in particular
\[
c_{k,\be} \le \max_A J_\be = \max_D E_\be \le e_{k,\be}+\eps.
\]
Since $\eps$ was arbitrarily chosen, we infer that $c_{k,\be} \le e_{k,\be}$. On the other hand, as $\cA_{k,\be} \subset \cB_{k,\be}$ and $E_\be=J_\be$ on $\cP_\be$, we have also that for any $A \in \cA_{k,\be}$
\[
\max_{A} J_\be = \max_{A} E_\be \ge e_{k,\be},
\]
whence the reverse inequality $e_{k,\be} \le c_{k,\be}$ follows.
\end{proof}

Now we proceed with the
\begin{proof}[Proof of Lemma \ref{lem: multiplicity}]
Suppose by contradiction that $\gamma(\mathcal{K}_c^+) \le p$. By Proposition \ref{prop: PS cond}, we have that $\mathcal{K}_c^+$ is compact. Then, by point ($v$) of Lemma \ref{lem: prop gamma}, there exists an open $\sigma$-invariant neighborhood $N$ of $\mathcal{K}_c^+$ in $\cE_\be$ such that $\gamma(\overline{N}) \le p$. Let $D \in \mathcal{B}_{j+p,\be}$ be arbitrarily chosen. Since $D \subset (D \setminus N) \cup \overline{N}$, by point ($ii$) of Lemma \ref{lem: prop gamma} we infer that $\gamma(D \setminus N) \ge j$, that is $D \setminus N \in \mathcal{B}_{j,\be}$. But then, by the definition of $e_j = c_j$, we have that $(D\setminus N) \cap E_{\be, e_{j}} \neq \emptyset$, where $E_{\be, e_{j}}$ is the superlevel set $\{E_\be \ge e_j\}$. For the closed $\sigma$-invariant set $F:=E_{\be, c_{j}} \setminus N$, we deduce that $F \cap D \neq \emptyset$ for every $D \in \mathcal{B}_{j+p,\be}$.
%This means that $F$ is dual to the $\sigma$-homotopy stable family with empty boundary $\mathcal{B}_{j+p,\be}$, and is clearly $\sigma$-invariant.

Let now $\{D_n\}$ be a minimizing sequence for $e_{j+p}$. Arguing as in the beginning of the proof of Proposition \ref{prop: PS}, we can suppose that $D_n \subset \cP_\be$ for every $n$. Moreover, arguing as in the proof of Lemma \ref{ex PS seq}, we can assume that any $(u,v) \in D_n$ is such that $u,v \ge 0$. Therefore, applying Theorem 7.2 in \cite{Ghou}, we deduce that there exists a Palais-Smale sequence $\{(\tilde u_n,\tilde v_n)\}$ for $E_\be$ on $\cS$ at level $e_{j+p}$ with the properties that
%$u_n^-,v_n^- \to 0$ a.e.\ in $\R^3$, and moreover
\begin{equation}\label{cond dist 1 11}
  \dist_{H^1(\R^3,\R^2)}((\tilde u_n,\tilde v_n), D_n) \to 0, \quad \text{and} \quad  \dist_{H^1(\R^3,\R^2)}((\tilde u_n,\tilde v_n), E_{\be, e_{j}} \setminus N) \to 0.
\end{equation}
Notice in particular that, since $D_n$ is compact, the first condition implies the existence of $(w_n,z_n) \in \mathcal{P}_\be$ with the properties ($a$) and ($b$) of Proposition \ref{rem: on s_n} (for the property ($b$), we can argue as in the proof of Proposition \ref{prop: PS}, using the fact that the level $e_j=c_j$ is finite), and satisfying also $w_n, z_n \ge 0$ a.e. in $\R^3$ for every $n$. Thus, Proposition \ref{rem: on s_n} ensures that $s_n=s_{(\tilde u_n,\tilde v_n)}$ tends to $0$ as $n \to \infty$, and that $(u_n,v_n) = s_n \star (\tilde u_n,\tilde v_n)$ is a Palais-Smale sequence for $J_\be$ at level $e_j$, with $(u_n,v_n) \in \cP_\be$ for every $n$, and $u_n^-,v_n^- \to 0$ a.e. in $\R^3$. %At this point we can proceed as in the proof of Proposition \ref{prop: PS}: introducing $s_n=s_{(\tilde u_n,\tilde v_n)}$ and $(u_n,v_n) = s_n \star (\tilde u_n,\tilde v_n)$, we can prove that $(u_n,v_n)$ is a Palais-Smale sequence for $J_\be$ at level $e_j$, with $(u_n,v_n) \in \cP_\be$ for every $n$, and and $u_n^-,v_n^- \to 0$ a.e. in $\R^3$.
Since $e_j = c_j \neq \ell$, Proposition \ref{prop: PS cond} implies that $(u_n,v_n) \to (u,v) \in \mathcal{K}^+_c$ strongly in $H^1$, up to a subsequence.

We are finally ready to reach a contradiction. On one side, by Lemma \ref{BaSo2.7}, we have $(\tilde u_n,\tilde v_n) = (-s_n) \star (u_n,v_n) \to (u,v) \in \mathcal{K}^+_c$ up to a subsequence, and in particular
\[
\dist_{H^1(\R^3,\R^2)}((\tilde u_n,\tilde v_n), \mathcal{K}_c^+) \to  0;
\]
but on the other side, by \eqref{cond dist 1 11}, there exists $C>0$ such that for every $n$ large
\[
\dist_{H^1(\R^3,\R^2)}((\tilde u_n,\tilde v_n), \mathcal{K}_c^+) \ge \inf_{(w,z) \in E_{\be, e_{j}} \setminus N}\dist_{H^1(\R^3,\R^2)}((w,z), \mathcal{K}_c^+) - o(1)  \ge C
\]
by definition of $N$, a contradiction.
\end{proof}

\section{Completion of the proof of the main results}\label{sec: end}

Combining Proposition \ref{prop: PS cond} and Lemmas \ref{ex PS seq} and \ref{lem: multiplicity}, the only fact that one has to check in order to obtain Theorem \ref{thm: main} is that
\[
\ell < c_{2,\be} \le \ldots \le c_{k+1,\be} < m_\be^\si \quad\text{for } \be < \be_k.
\]
Only the first inequality $\ell< c_{2,\be}$ is not obvious. This is a consequence of the following statement.

\begin{lemma}\label{lem: genus sub}
There exists $\delta>0$ such that the nonnegative closed sublevel set
\[
  J_{\mathcal{P}_\be^+}^{\ell+\delta}
   := \left\{(u,v) \in \cP_\be: \ u,v\ge 0 \ \text{a.e. in $\R^3$},\ J_\be(u,v) \le \ell +\delta\right\}
\]
has genus $1$.
\end{lemma}

\begin{proof}
We claim that for every $\eps>0$ there exists $\delta>0$ such that $(u,v) \in \cP_\be$, $u,v \ge 0$ a.e.\ in $\R^3$ and $J_\be(u,v) \le \ell + \delta$ implies
\begin{equation}\label{cl 22 gen}
  \text{either} \quad \|u- w_0\|_{H^1} + \|v\|_{\mathcal{D}^{1,2}}  < \eps, \quad
  \text{or} \quad \|v- w_0\|_{H^1} + \|u\|_{\mathcal{D}^{1,2}} < \eps.
\end{equation}
If this claim were false, then we would find $\eps>0$ and a sequence $\{(w_n, z_n)\} \subset \cP_\be$, $w_n,z_n \ge 0$ a.e. in $\R^3$, such that $J_\be(w_n,z_n) \to \ell$ and
\begin{equation}\label{011106}
  \text{both} \quad \|w_n- w_0\|_{H^1} + \|z_n\|_{\mathcal{D}^{1,2}} \ge \eps, \quad
  \text{and} \quad \|z_n- w_0\|_{H^1} + \|w_n\|_{\mathcal{D}^{1,2}} \ge \eps.
\end{equation}
Since $\{(w_n,z_n)\} \subset \cP_\be$, we have
\[
E_\be(w_n,z_n) =J_\be(w_n,z_n) = \frac16 \|(w_n,z_n)\|_{\cD^{1,2}}^2 \to \ell,
\]
and hence $\{(w_n,z_n)\}$ is a bounded minimizing sequence for $E_\be$ on $\cE_\be$, see Lemmas \ref{inf ray} and \ref{lem: inf P}. By Ekeland's variational principle, there exists then a Palais-Smale sequence $\{(\tilde u_n,\tilde v_n)\}$ for $E_\be$ on $\cE_\be$, with the property that $\|(\tilde u_n,\tilde v_n)-(w_n,z_n)\|_{H^1} \to 0$ as $n \to \infty$. As a consequence, letting $s_n:= s_{(\tilde u_n,\tilde v_n)}$ and $(u_n,v_n) := s_n \star (\tilde u_n,\tilde v_n)$, by Proposition \ref{rem: on s_n} we have that $\{(u_n,v_n)\}$ is a Palais-Smale sequence at level $\ell$ for $J_\beta$ on $\cS$, with $(u_n,v_n) \in \cP_\be$ for every $n$, $u_n^-,v_n^- \to 0$ a.e. in $\R^3$, and $s_n \to 0$ as $n \to \infty$.

In order to describe the asymptotic behavior of $\{(u_n,v_n)\}$, we observe at first that Proposition \ref{prop: PS cond} is applicable, and hence one of the alternatives ($i$) and ($ii$) holds.

Let us prove that ($ii$) cannot occur. By \eqref{011106} and the fact that $\|(\tilde u_n,\tilde v_n)-(w_n,z_n)\|_{H^1} \to 0$, we have
\begin{equation}\label{011105}
\text{both} \quad \|\tilde u_n- w_0\|_{H^1} + \|\tilde v_n\|_{\mathcal{D}^{1,2}} \ge \frac34\eps, \quad
  \text{and} \quad \|\tilde v_n- w_0\|_{H^1} + \|\tilde u_n\|_{\mathcal{D}^{1,2}} \ge \frac34\eps.
\end{equation}
Now, if alternative ($ii$) holds, we have for instance $u_n \to w_0$ strongly in $H^1(\R^3)$ and $v_n \to 0$ strongly in $\cD^{1,2}(\R^3)$. But using the fact that $s_n \to 0$, we deduce that also $\tilde u_n = -s_n \star u_n  \to w_0$ strongly in $H^1$, and $\|\tilde v_n\|_{\cD^{1,2}} = e^{-s_n} \|v_n\|_{\D^{1,2}} \to 0$, in contradiction with \eqref{011105}.

This shows that necessarily alternative ($i$) in Proposition \ref{prop: PS cond} holds true, i.e. $(\tilde u_n,\tilde v_n) \to (u_\be,v_\be)$ strongly in $H^1(\R^3,\R^2)$, where $(u_\be, v_\be)$ is a positive solution to \eqref{system}, and achieves the minimum of $E_\be$ on $\cE_\be$. Both $u_\beta$ and $v_\beta$ are strictly positive in $\R^3$ by the strong maximum principle, and hence $\int_{\R^3} u_\beta^2 v_\beta^2 >0$. But then, recalling \eqref{expE},
\[
  \ell = E_\beta(u_\beta,v_\beta) > E_0(u_\beta,v_\beta) \ge \inf_{(u,v) \in \cE_0} E_0(u,v) = m_0 = \ell,
\]
a contradiction again. This proves the validity of claim \eqref{cl 22 gen}.

Let now $\eps>0$ so small that $\|u- w_0\|_{H^1} < \eps$ implies $\|u\|_{\mathcal{D}^{1,2}} > \eps$. The above argument shows that for any such $\eps$ there exists a small positive $\delta$ such that $J_{\mathcal{P^+}}^{\ell+\delta} \subset D$, where
\begin{align*}
  D &:= \left\{(u,v) \in \cP_\be \left|\begin{array}{l}
          \text{either} \quad \|u- w_0\|_{H^1} + \|v\|_{\mathcal{D}^{1,2}} \le \eps, \\
          \text{or} \quad \|v- w_0\|_{H^1} + \|u\|_{\mathcal{D}^{1,2}}  \le \eps \end{array} \right. \right\}.
\end{align*}
By definition, we have that $D=D_1 \cup D_2$ with
\[
  D_1:= \left\{(u,v) \in \cP_\be : \  \|u- w_0\|_{H^1} + \|v\|_{\mathcal{D}^{1,2}}  \le \eps\right\},
  \quad \quad D_2 := \sigma(D_1),
\]
and $D_1 \cap D_2 = \emptyset$ by the choice of $\eps$. Therefore, $D$ is the disjoint union of two closed sets with $D_2 = \sigma(D_1)$, which implies $\gamma(D)=1$. By the monotonicity property of the genus, point ($i$) in Lemma \ref{lem: prop gamma}, the thesis follows.
\end{proof}

\begin{proof}[Conclusion of the proof of Theorem \ref{thm: main}]
We omit the dependence of the quantities with respect to $\beta$, which is fixed throughout this proof. Since we already know that $c_k \le c_{k+1}$ for every $k \ge 1$, in order to show the validity of the theorem we can simply prove that $c_2>c_1 = \inf_{\cP} J = \ell$. By contradiction, suppose that $c_2 = c_1$. Then there exists a sequence $\{A_n\} \subset \mathcal{A}_2$ with $\sup_{A_n} J \to \ell$, and in particular $\sup_{A_n} J \le \ell+\delta$ for every $n$ sufficiently large. Let us consider the set $|A_n|$, defined in \eqref{mod set}. By point ($iii$) of Lemma \ref{lem: prop gamma}, we know that $\gamma(|A_n|) \ge \gamma(A_n) \ge 2$ for every $n$; on the other hand, observing that $J(u,v) = J(|u|, |v|)$ for every $(u,v) \in \mathcal{P}$, we deduce that $|A_n| \subset J_{\cP^+}^{\ell+\delta}$, and hence by point ($i$) of Lemma \ref{lem: prop gamma} together with Lemma \ref{lem: genus sub} we have also $\gamma(|A_n|) \le \gamma(J_{\cP^+}^{\ell+\delta}) =1$, a contradiction.

It remains to show that, for $\beta \le -\mu = -1$, we have $J(u_{k},v_{k}) \to +\infty$ as $k \to +\infty$. Let us introduce the generalized Morse index $m_\cP(u,v)$ of $J|_{\mathcal{P}}$ in a critical point $(u,v)$ as the dimension of the negative and null eigenspace of the linearized operator $d^2 J|_{\mathcal{P}}(u,v)$. Similarly we write $m(u,v)$ for the generalized Morse index of $J$ on $\cS$. Observe that these differ by at most one because $\cP$ is a codimension one submanifold of $\cS$. In fact, $m(u,v)=m_\cP(u,v)+1$, because the path $t\mapsto t*(u,v)$ is transversal to $\cP$, and $J$ achieves its maximum along the path at $t=0$. By \cite[Corollary 10.5]{Ghou}, the min-max characterization of $(u_k,v_k)$ yields an estimate on the Morse index (on the line of \cite{BahLio, So}), and to be precise we have that $m_\cP(u_{k},v_{k}) \ge k$. Now, let us assume that $c_k \to \bar c<+\infty$. Then the sequence $(u_{k},v_{k})$ is a Palais-Smale sequence for $J|_{\mathcal{P}}$ at level $\bar c > \ell$ made of positive functions. Hence, by Proposition \ref{prop: PS cond}, it is convergent to a limit $(\bar u, \bar v)$, which is a positive radial solution to \eqref{system} and has therefore finite generalized Morse index. This can be seen as follows. The gradient $\nabla_uJ:\cS\to TS_a$ with respect to the standard scalar product in $H^1$ is given by
\[
  \nabla_uJ(u,v) = (-\De+1)^{-1}\left(-\De u-u^3-\be uv^2-\la_1(u,v)u\right)
\]
where $\la_1(u,v)\in\R$ is determined by the equation $\nabla_uJ(u,v) \in T_uS_a$, i.e.\ $\int_{\R^3}u\cdot\nabla_uJ(u,v) = 0$. Similarly we obtain
\[
  \nabla_vJ(u,v) = (-\De+1)^{-1}\left(-\De v-v^3-\be u^2v-\la_2(u,v)v\right).
\]
Therefore the Hessian of $J$ at the critical point $(\bar u, \bar v)$ in the direction $(\phi,\psi)\in T_{(\bar u, \bar v)}\cS$ is computed as follows:
\[
\begin{aligned}
D^2J(\bar u,\bar v)[(\phi,\psi),(\phi,\psi)]
 &= \int_{\R^3} \left(|\nabla\phi|^2+|\nabla\psi|^2-\la_1(\bar u,\bar v)\phi^2-\la_2(\bar u,\bar v)\psi^2\right)\\
 &\hspace{1cm} -\int_{\R^3} \left((3u^2+\be v^2)\phi^2-(3v^2+\be u^2)\psi^2-4\be uv\phi\psi\right).
\end{aligned}
\]
We have used the fact that $\int_{\R^3}\bar u\phi=0=\int_{\R^3}\bar v\psi$. Here $\la_1(\bar u,\bar v),\la_2(\bar u,\bar v)$ are the Lagrange multipliers of the solution $(\bar u, \bar v)$, hence they are negative. Therefore the first integral above is strictly positive definite, whereas the second integral defines a quadratic form on $T_{(\bar u, \bar v)}\cS \subset H^1_{\mathrm{rad}}\times H^1_{\mathrm{rad}}$ which is even defined, and continuous, on $L^4\times L^4$. Since the embedding of $H^1_{\mathrm{rad}}(\R^3)$ into $L^4(\R^3)$ is compact, the negative eigenspace and the kernel of $D^2J(\bar u,\bar v)$ must be finite-dimensional, and  $D^2J(\bar u,\bar v)$ is strictly positive definite on a subspace $X^+\subset T_{(\bar u, \bar v)}\cS$ of finite codimension, i.e.\
$D^2J(\bar u,\bar v)[(\phi,\psi),(\phi,\psi)] \ge c\|(\phi,\psi)\|^2$ for some $c>0$, and all $(\phi,\psi)\in X^+$. This, however, contradicts the fact that $m(u_k,v_k) \to +\infty$, and completes the proof.
\end{proof}

\begin{proof}[Proof of Theorem \ref{thm: phase sep}]
We can proceed exactly as in Section 3.4 in \cite{BaSo}.
\end{proof}

\end{document}